\documentclass[a4paper, reqno]{amsart}
\usepackage[margin=1in]{geometry}
\usepackage{enumerate}
\usepackage{amsthm}
\usepackage{amsmath}
\usepackage[mathscr]{euscript}
\usepackage{amssymb}
\usepackage[noadjust]{cite}
\usepackage{comment}
\usepackage{color}
\usepackage{filecontents}
\usepackage{tabularx}

\usepackage{tikz}
\usetikzlibrary{calc,positioning,fit,backgrounds}
\pgfdeclarelayer{background}
\pgfsetlayers{background,main}
\usepackage{caption}

\usepackage{color} 
\usepackage{hyperref}
\hypersetup{
	colorlinks=true,
	linkcolor=blue,
}
\usepackage{graphicx}
\usepackage{hyperref}

\usepackage{amssymb,amsthm,amsmath,mathrsfs}
\usepackage{graphicx}
\usepackage{multicol}
\usepackage{color}
\usepackage{tikz}
\usepackage{tikz-cd}

\newtheorem*{genericthm*}{\thistheoremname}
\newcommand{\thistheoremname}{???}
\newcounter{genericthm}

\newenvironment{namedthm*}[1]
  {\renewcommand{\thistheoremname}{#1}%
   \refstepcounter{genericthm}%
   \begin{genericthm*}}
  {\end{genericthm*}}

\newtheorem{thm}{Theorem}[section]
\newtheorem{lem}[thm]{Lemma}
\newtheorem{prop}[thm]{Proposition}
\newtheorem{cor}[thm]{Corollary}

\theoremstyle{definition}

\newtheorem{defn}[thm]{Definition}

\theoremstyle{remark}
\newtheorem{ex}[thm]{Example}
\newtheorem{rem}[thm]{Remark}

\begin{document}
\title{Hilbert transforms on Coxeter groups and groups acting on buildings}
	
\author[Lu and Xia]{Xiao-Qi Lu \and Runlian Xia}
	
\address[Xiao-Qi Lu]{School of mathematics and Statistics, University of Glasgow, University Place, Glasgow G12 8QQ, UK}
\email{Xiaoqi.Lu@glasgow.ac.uk}

\address[Runlian Xia]{School of mathematics and Statistics, University of Glasgow, University Place, Glasgow G12 8QQ, UK}
\email{Runlian.Xia@glasgow.ac.uk}

\keywords{von Neumann algebras, non-commutative Cotlar identities, Coxeter groups, buildings satisfying the nested condition}

\subjclass[2020]{Primary: 46L07, 46L10, 46L52}

\begin{abstract}
In this paper, we study Hilbert transforms and their boundedness on $L_p$-spaces associated with Coxeter groups and groups acting on buildings. We establish new models for Hilbert transforms on these groups through the geometric objects they act on, and we show that these Hilbert transforms satisfy a Cotlar identity which was developed in earlier work of Mei and Ricard, and that of  Gonz\'alez-P\'erez, Parcet and the second named author, thus we conclude the $L_p$-boundedness. We manage to solve the problem firstly for the case when the  Coxeter group satisfies a certain nested condition, and then extend it to any finitely generated Coxeter groups and groups that admit actions on buildings satisfying the nested condition. 
Our results give the first example of groups with property (F$\mathbb{R}$) on which $L_p$-bounded Hilbert transforms can be defined, and generalize the work of Gonz\'alez-P\'erez, Parcet and the second named author for groups acting on simplicial trees. 
\end{abstract}

\maketitle

\section*{introduction}
Fourier multipliers and their boundedness on $L_p(\mathbb{R}^n)$ $(n\in \mathbb{N}, 1<p<\infty)$ play a prominent role in harmonic analysis. 
One of the most important Fourier multipliers is the Hilbert transform on $\mathbb{R}$ (or directional Hilbert transforms in the higher-dimensional case). Hilbert transforms are well known for their role in describing the boundary behaviour of harmonic conjugates on the upper half-plane, and their boundedness on $L_p(\mathbb{T})$ is equivalent to the convergence of Fourier series in the same space. Given $f\in L_p(\mathbb{R})$ and any $x\in \mathbb{R}$, the Hilbert transform is defined by
$$Hf(x) := \frac{1}{\pi}\lim_{\varepsilon\to 0^+}\int_{|x-y|>\varepsilon} \frac{f(y)}{x-y} dy.$$
 Taking the Fourier transform, the Hilbert transform can also be viewed as a  Fourier multiplier \cite[\S 5.2]{king2009hilbert}; namely,
$$(Hf)^{\wedge} (t) = m(t)\widehat{f}(t) = -i {\rm sgn}(t)\widehat{f}(t).$$
In 1928, M. Riesz proved the $L_p$-boundedness ($1<p<\infty$) of the Hilbert transform, using Green's theorem as a key ingredient (see \cite[\S 7.1]{king2009hilbert}). In 1955, Cotlar gave a comparatively simple proof of the same result in \cite{cotlar1955unified}. He showed that $Hf$ satisfies the following \textit{Cotlar identity}:
\begin{equation}\label{equation;Cotlar-classical}
	(Hf)^2 = 2H(fHf) - H(Hf^2).
\end{equation}
Since $H$ is trivially bounded on $L_2(\mathbb{R})$ by the Plancherel formula, by induction, the $L_p$-boundedness follows for any $p = 2^{i}$ ($i\in \mathbb{N}$). Then the Riesz-Thorin interpolation theorem and the duality imply the $L_p$-boundedness of the Hilbert transform for $1< p < \infty$.

The main goal of this paper is to study analogues of Hilbert transforms on locally compact groups arising from their actions on geometric structures that admit half-space structures. More precisely, we focus on structures such as Coxeter complexes and buildings. This leads to interesting interplays between geometry and harmonic analysis. We show that, for the action of a Coxeter group on its Coxeter complex, one can define a corresponding natural Hilbert transform on the Coxeter group. In particular, if the Coxeter group satisfies a certain nested condition relative to a generator, the Hilbert transform gives rise to an $L_p$-bounded Fourier multiplier. Since any finitely generated Coxeter group admits a torsion-free subgroup of finite index that satisfies the nested condition, we extend this boundedness result to all finitely generated Coxeter groups. Building on our construction for Coxeter groups, we further generalize the result to groups acting on buildings satisfying the nested condition, which includes a broader class of groups, 
such as graph products of groups. 

The study of Fourier multipliers  on groups and their boundedness on the associated non-commutative $L_p$-spaces arises naturally as a non-commutative counterpart of the classical theory. This line of research dates back to the work of Haagerup \cite{Haag79}, who investigated the metric approximation property of the reduced $C^*$-algebras and von Neumann algebras of free groups using Fourier multipliers. Following Haagerup's work, 
completely bounded Fourier multipliers on simple Lie groups have been studied by Cowling, de Canni\`ere and Haagerup in \cite{CowlingHaag89, DecanHaag85} in connection with the characterization of the completely bounded approximation property (CBAP) of the corresponding group von Neumann algebras. Since the non-commutative $L_p$-spaces associated with group von Neumann algebras are operator spaces, one may also consider the CBAP for these spaces. 
A remarkable result in this direction was obtained by Lafforgue and de la Salle \cite{LafdelaSalle11}, where they showed that non-commutative $L_p$-spaces on ${\rm SL}_n(\mathbb{R})$ ($n\geq 3$) fail to have CBAP whenever $p$ exceeds a threshold depending on $n$. This result is interesting in view of the $W^*$-superrigidity problem proposed by Connes in \cite{connes1982classification}, which highlights the role of harmonic-analytic methods in the study of operator algebras. In recent years, H\"ormander-Mikhlin type $L_p$-bounded Fourier multipliers on higher-rank simple Lie groups have been extensively studied by Conde-Alonso, de la Salle, González-Pérez, Parcet, Ricard and Tablate in \cite{ParRidelaS22, ConGonParTab24}, and a non-commutative Calder\'on-Torchinsky theorem has been developed on semisimple Lie groups by Caspers in \cite{Cas23}. Motivated by the importance of Hilbert transforms  in the construction of approximate identities on $L_p(\mathbb{R}^n)$, constructing Hilbert transforms for new classes of groups might lead to new $L_p$-bounded Fourier multipliers which could shed light on the CBAP of the associated non-commutative $L_p$-spaces.


For a locally compact group $G$, consider its left regular representation $\lambda : G \rightarrow  B(L_2(G))$. It induces a non-degenerate representation of $L_1(G)$ on $L_2(G)$, which we also denote by $\lambda$ without ambiguity. We write $\lambda(f)= \int_{G}f(g)\lambda (g)d\mu(g)$, for $f\in L_1(G)$.  
The \textit{group von Neumann algebra} $\mathcal{L}G\subseteq B(L_2(G))$ of $G$ is defined as the weak closure of $\lambda(L_1(G))$. 
When $G$ is separable and unimodular,  there is a normal, faithful, semifinite trace on $\mathcal{L}G$, with respect to which the non-commutative $L_p$-spaces $L_p(\mathcal{L}G)$ for $1\le p \le\infty$ are defined (see \cite{pisier2003non}). 
Let $m: G\to \mathbb{C}$ be a bounded function and $C_c(G)$ be the space of all continuous functions with compact support. 
The Fourier multiplier with symbol $m$ is an operator 
 $T_m: \lambda(C_c(G))\subseteq \mathcal{L}G \to \mathcal{L}G$  given by the linear extension of 
\begin{equation}\label{equation; Fourier transform on groups}
	T_m(\lambda(f)) := \int_G m(g)f(g)\lambda (g) \,dg.
\end{equation}
In the classical setting, we can use Calder\'{o}n-Zygmund theory to deal with the $L_p$-boundedness of Fourier multipliers. There have been many works on the counterpart of Calder\'{o}n-Zygmund theory in the non-commutative setting, for example, \cite{CCAP22, JMPX21, Parcet09}. 
However, the generalization in the group von Neumann algebras setting is not yet known. On the frequency side, the H\"ormander-Mikhlin theory for Fourier multipliers developed in \cite{ParRidelaS22, ConGonParTab24} relies on the fact that the groups studied there admit Kazhdan's property (T). So far, no H\"ormander-Mikhlin theory is known for groups without Kazhdan's property (T). 
For Hilbert transforms, 
Mei and Ricard initiated the study of a non-commutative analogue of  (\ref{equation;Cotlar-classical}) and constructed Hilbert transforms on amalgamated free products that satisfy a non-commutative Cotlar identity in \cite{mei2017free}. Later it was pointed out in \cite{gonzalez2022non-commutative} by Gonz\'alez-P\'erez, Parcet and the second named author that  
there exists another equivalent characterization of the Cotlar identity defined in \cite{mei2017free} in terms of the symbol of the Hilbert transform. In particular, in the case of $\mathbb{R}$, the classical Cotlar identity \eqref{equation;Cotlar-classical} can be reformulated as 
\begin{equation}\label{equation;Cotlar-m-classical}
	(m(-s)-m(t))(m(s+t) - m(s)) = 0, \quad \text{ for  a.e. } s, t \in \mathbb{R}.
\end{equation}
A main theorem of \cite{gonzalez2022non-commutative}, stated below, shows that 
if  one can define a Fourier multiplier  on a group whose symbol satisfies a Cotlar identity, then it is bounded on $L_p({\mathcal{L} G})$ for every $1< p< \infty$. 

\begin{namedthm*}{Theorem CI}(\cite[Theorem A]{gonzalez2022non-commutative})
\label{prop; Lp-bounded-introduction}
Let $G$ be a locally compact unimodular group, let $G_0 \subseteq G$ be a closed subgroup, and let $m:G \rightarrow \mathbb{C}$ be a left $G_0$-invariant bounded function. If $m$ satisfies the condition
\begin{equation}\label{eq: Cotlar}
\tag{$\widehat{\text{Cotlar}}$}
 (m(g)- m(h))(m(g^{-1}h) - m(g^{-1})) = 0, \quad \text{ for a.e. } g\in G\setminus G_0, h\in G,  
\end{equation}
then for every $1< p< \infty$, one has
	$$\|T_m:L_p({\mathcal{L} G})\to L_p(\mathcal{L} G)\|\lesssim \left(\frac{p^2}{p-1}\right)^\alpha \| m\|_\infty  \quad \text{ with } 
	\alpha = \log_2(1+\sqrt{2}).
    $$
\end{namedthm*}
A geometric model for Hilbert transforms on groups arising from their non-trivial actions on tree-like structures ($\mathbb{R}$-trees) was introduced in \cite{gonzalez2022non-commutative}, and it was also shown that those Hilbert transforms satisfy \eqref{eq: Cotlar}. 
In the following, we summarize the model in \cite{gonzalez2022non-commutative} for the case where the tree-like structure is a simplicial tree.  Let $G \curvearrowright X$ be an isometric action on a tree, let $x_0 \in X$ be a vertex,
  and let $\widetilde{m}: X \to \mathbb{C}$ be a bounded  function satisfying the following conditions:
  \begin{enumerate}
    \item[(i)] $\widetilde{m}$, restricted to $X \setminus \{x_0\}$, is constant along connected components.
    \item[(ii)] $\widetilde{m}$ is invariant under the action ${\rm Stab}_{x_0} \curvearrowright  X \setminus \{x_0\}$.
  \end{enumerate}
  Then the symbol $m: G \to \mathbb{C}$ of the associated Hilbert transform is defined as
  \begin{equation}\label{eq: Multiplier on trees}
    m(g) = \widetilde{m}(g \cdot x_0).  
  \end{equation}
This model applies to many interesting groups, including left-orderable groups and fundamental groups of graphs of groups, such as amalgamated free products and HNN extensions. However, the geometric model used in  \cite{gonzalez2022non-commutative} does not apply to groups with property (F$\mathbb{R}$),
that is, groups for which every action on an  $\mathbb{R}$-tree admits a global fixed point. For example, many Coxeter groups and all groups with Kazhdan’s property (T) have  property (F$\mathbb{R}$). 
This limitation provides part of the motivation for the present paper. In order to treat Coxeter groups, and more generally, groups acting on buildings, new models are required.

 \subsection*{Hilbert transforms on Coxeter groups}
 Coxeter groups with every pair of generators  
satisfying a non-trivial relation have property (F$\mathbb{R}$) (an application of Helly’s theorem for $\mathbb{R}$-trees, see \cite[Section 1.2]{chatterji2010kazhdan}).  Nevertheless, every Coxeter group admits natural actions on its Coxeter complex and  Davis complex (for precise definitions, see Section \ref{section;Cotlar-proof-Coxeter}), both of which provide rich geometric and topological structures for the study of Hilbert transforms.

Coxeter introduced Coxeter groups as a generalization of reflection groups and completely classified finite Coxeter groups in the 1930s. The presentation of infinite Coxeter groups was systematically studied by Tits in the 1960s. These groups are connected with many areas of mathematics, including graph theory and Lie theory.  Let $S$ be a finite set. Let ${\mathrm{m}} : S\times S\to \{1,2,\ldots ,\infty\}$ be a map such that  $\mathrm{m}_{st}=1$ if $ s=t$, $\mathrm{m}_{st}\ge 2$ if $s\ne t$, and $\mathrm{m}_{st} = \mathrm{m}_{ts}$, where $\mathrm{m}_{st}=\mathrm{m}(s,t)$ for $s, t\in S$. 
We call the matrix $M :=(\mathrm{m}_{st})_{S\times S}$ the \textit{Coxeter matrix}.
Given a Coxeter matrix $M :=(\mathrm{m}_{st})_{S\times S}$, the associated \textit{Coxeter group} $W$ is defined by the following presentation: 
$$W := \langle S:  (st)^{\mathrm{m}_{st}} = 1\rangle.$$
By convention, if $\mathrm{m}_{st}=\infty$, then no relation of the form $(st)^{r}=1$ for any integer $r\geq 2$ should be imposed. We call the pair $(W, S)$ a \textit{Coxeter system} and the relations $(st)^{\mathrm{m}_{st}} = 1$ in the presentation of $W$ the \textit{braid relations}. It is straightforward to verify that for every $s\in S$, $s^2=1$. Moreover, for $s,t\in S$, $(st)^{\mathrm{m}_{st}} = 1$ if and only if $(ts)^{\mathrm{m}_{st}} = 1$.

Every Coxeter group admits a geometric representation called the \textit{Tits representation} on a  real vector space, and every generator corresponds to a reflection on that vector space. Let $V$ be a real vector space with basis $(e_s)_{s\in S}$. Define a bilinear form on $V$ by $$\langle e_s, e_t \rangle = -\cos (\frac{\pi}{\mathrm{m}_{st}}).$$
Note that when $\mathrm{m}_{st} = \infty$, we set $\langle e_s, e_t \rangle = -1$.
The Tits representation $\rho: W \longrightarrow {\rm GL}(V)$ is defined by 
$$
\rho(s) x= x-2\langle e_s, x \rangle e_s, \quad \;\; \text{ for any }s\in S \text{ and } x\in V.
$$
For any $s\in S$, $\rho(s)$ is a reflection on $V$; it sends $e_s$ to $-e_{s}$ and fixes the hyperplane $H_s=\{x\in V: \langle e_s, x \rangle=0\}$. One can check that the representation defined above is a faithful representation. In other words, we may identify $W$ with a subgroup of ${\rm GL}(V)$ generated by reflections.

Coxeter complexes and Davis complexes can be embedded into the real vector space $V$ constructed above. In these complexes, we have a type of  intrinsic geometric structure called walls (or hyperplanes in $V$), which split the complex into two roots (or half-spaces). We now give a rough description of our model for Hilbert transforms on Coxeter groups using the root structure; the precise definition can be found in Section \ref{section: Cotlar}. 
Given a generator $s\in S$, there is an associated wall $H_s$, which determines two roots $H_s^+$ and  $H_s^-$. Note that every element of a Coxeter group is in one-to-one correspondence with a chamber in the Coxeter complex and the Davis complex. Thus, Coxeter groups admit natural actions on these complexes that map chambers to chambers.  If we denote by $K$ the chamber in the Coxeter complex corresponding to the identity element of $W$ (usually referred to as the fundamental chamber), then we define a multiplier $m_W^s: W \rightarrow \mathbb{C}$ by
\begin{equation*}
		m_W^s(g) := \left\{\begin{array}{l}
			1, \quad \text{if } \; gK\in H_s^{+},\\
			-1, \quad \text{if } \; gK\in H_s^{-}.
		\end{array} \right.
\end{equation*}
We show that whenever the Coxeter group 
satisfies a nested condition relative to $s$ (a precise definition is given in Definition \ref{Defn;Nested-condition}), the multiplier defined above satisfies \eqref{eq: Cotlar} for every $g\notin {\rm Stab }_{H_s^+}$.
We also show, in Proposition \ref{prop;equi-char-Cotlar-nested}, the assumption that $W$ satisfies a nested condition relative to $s$ is optimal, in the sense that the multiplier defined above satisfies \eqref{eq: Cotlar} relative to ${\rm Stab }_{H_s^+}$ if and only if $W$ satisfies a nested condition relative to $s$.

\begin{namedthm*}{Theorem A}\label{theoremA}
    Let $(W,S)$ be a Coxeter system. Assume that there exists $s\in S$ such that $W$ satisfies the nested condition relative to $s$. 
   Set ${\rm Stab}_{ H_s^+} = \{g\in W: g\cdot H_s^+ = H_s^+\}$. 
Then the following statements hold:
\begin{enumerate}
    \item [(i)] The multiplier $m_W^s$ satisfies \eqref{eq: Cotlar}, i.e.
    $$(m_W^s(g) - m_W^s(h))(m_W^s(g^{-1}h) -m_W^s(g^{-1})) = 0, \text{ for any }g\in W\setminus {\rm Stab}_{ H_s^+} \text{ and } h\in W.$$

 \item [(ii)] The multiplier $m_W^s$ is left ${\rm Stab}_{ H_s^+}$-invariant, i.e. for any $h'\in {\rm Stab}_{ H_s^+}$ and $h\in W$, we have $m_W^s(h'h) = m_W^s(h)$. 
\end{enumerate}
 Therefore, $T_{m_W^s}$ gives a Fourier multiplier bounded on $L_p(\mathcal{L} W)$ for every $1< p< \infty$.
\end{namedthm*}

We will give a characterization in terms of the braid relations for Coxeter groups that satisfy the nested condition in Proposition \ref{prop;W-satisfies-NestC-equi-to-WRA}. In particular, this class includes all right-angled Coxeter groups and the word-hyperbolic Coxeter group $\mathrm{PGL}(2,\mathbb{Z})$. However, most Coxeter groups with property (F$\mathbb{R}$) fail to satisfy the nested condition. 
Therefore, to deal with Coxeter groups in full generality, we will need new techniques. 
From the Tits representation, Coxeter groups can be viewed as subgroups of $\mathrm{GL}(n,\mathbb{R})$. Then by Selberg's lemma \cite{selberg1960discontinuous}, every Coxeter group $W$ admits a torsion-free subgroup $W_0$ of finite index. 
It turns out that an arbitrary normal torsion-free subgroup $W_0$ (without loss of generality, we may assume that $sW_0s=W_0$ for a generator $s$) also satisfies the nested condition relative to $s$ with respect to roots in Davis complexes, as required in \ref{theoremA}. Then $L_p$-bounded Hilbert transforms can be defined on such finite-index subgroups. By extending these Hilbert transforms to the whole group, we will be able to construct $L_p$-bounded Hilbert transforms on any Coxeter group.  
Suppose that the index of $W_0$ in $W$ is $k\in \mathbb{N}$ and that $W= \mathop{\bigsqcup}\limits_{i=1}^kw_iW_0$, where $w_i\in W$ for $i = 1,\ldots ,k$. For any $g= w_ih$  with $h\in W_0$, define $m$ on $W$ by
\begin{equation*}
		m(g) := \left\{\begin{array}{l}
			1, \quad \text{if} \; hK\in H_s^+,\\
			-1, \quad \text{if} \; hK\in H_s^-,
		\end{array}\right.
\end{equation*}
where $H_s^\pm$ are the roots associated with $s$ and $K$ is the fundamental chamber in the Davis complex of $W$. We refer the reader to the discussion immediately preceding and following Lemma \ref{cor;torsion-free-cotlar} for the precise definition of the multiplier $m$.
\begin{namedthm*}{Theorem B}\label{theoremB}
Let $(W,S)$ be a Coxeter system. The multiplier $m$ defined above  gives a Fourier multiplier $T_m$ bounded on $L_p(\mathcal{L} W)$ for every $1< p< \infty$.
\end{namedthm*}

\subsection*{Hilbert transforms on groups acting on buildings}
Coxeter groups are also closely related to buildings (see Subsection \ref{subsec: buildings} or \cite{abramenko2008buildings, ronan2009lectures, thomas2018geometric} for further details). 
Roughly speaking, a building is said to be of type $(W,S)$ for a Coxeter system $(W,S)$ if it is obtained by gluing together subcomplexes (or apartments), each of which is isomorphic to the Coxeter complex of $W$. On the other hand, buildings can be described as a non-empty set $\Delta$ equipped with a Weyl distance function $\delta: \Delta \times \Delta\rightarrow W$ satisfying  certain axioms (see Definition \ref{def: Buildings}). The elements of $\Delta$ are called chambers. 
Now we consider a locally compact group $G$ that admits actions on $(\Delta, \delta)$ preserving the distance $\delta$. Given a chamber $C_0$ and a generator $u\in S$, define a function $m_G^u$ on $G$ by
\begin{equation*}
    m_G^u(g) := \left\{\begin{array}{l}
		1, \quad \text{if} \; gC_0\in \phi(H_u^+),\\
		-1, \quad \text{if} \;  gC_0\in \phi(H_u^-), 
	\end{array} \right.
\end{equation*}
where $\phi$ is an isometry from the Coxeter complex into the building whose image contains $C_0$ and $gC_0$ such that $\phi(K) = C_0$.
We now state our third main result.
 
\begin{namedthm*}{Theorem C}\label{theoremC}
Let $(\Delta, \delta)$ be a building of type $(W,S)$, where $W$ satisfies the nested condition relative to $u\in S$. Let $G$ be a unimodular locally compact group acting on $\Delta$ isometrically. Fix a chamber $C_0\in \Delta$,  we assume additionally that the orbit map $G \rightarrow \Delta:  g\mapsto gC_0$ is continuous when $\Delta$ is equipped with the discrete topology.
 Set
 $G_0 := \{g\in G: \delta(C_0, gC_0)\in W_{T_u}\}$ where $T_u:= \{s\in S: \mathrm{m}_{su} =2\}$. 
Then  the following statements hold:
\begin{enumerate}
    \item[(i)] $G_0$ is  an  open subgroup of $G$, and $m_G^u$ satisfies \eqref{eq: Cotlar}, i.e.  
    $$(m_G^u(g) - m_G^u(h))(m_G^u(g^{-1}h) -m_G^u(g^{-1})) = 0, \text{ for any }g\in G\setminus G_0 \text{ and } h\in G.$$
 \item[(ii)] The multiplier $m_G^u$ is left $G_0$-invariant. 
\end{enumerate}
Therefore $T_{m_G^u}$ gives a bounded Fourier multiplier on $L_p(\mathcal{L} G)$ for every $1< p< \infty$. 
\end{namedthm*}

As an application, we obtain Hilbert transforms on any graph products of groups, which we will illustrate in Subsection \ref{subsection3.5}.

\section{Preliminaries}\label{section;Cotlar-proof-Coxeter}

\subsection{Non-commutative $L_p$-spaces}
 Let $(\mathcal{M},\tau)$ be a von Neumann algebra equipped with a \textit{normal, faithful, semifinite trace} (simply denoted by \textit{n.f.s. trace}) $\tau$, and let $\mathcal{S}_{\mathcal{M}}$ denote the linear span of the set 
$$\{x\in \mathcal{M}_+: \tau(s(x))<\infty\},$$ 
where $s(x)$ is the support projection of $x$ (i.e. the smallest projection such that $s(x)x = x$). The non-commutative $L_p$-space $L_p(\mathcal{M},\tau)$ ($1\le p<\infty$) is defined as the completion of $\mathcal{S}_{\mathcal{M}}$ under $\|\cdot\|_p$, where 
$$\|x\|_p := \tau(|x|^p)^{\frac{1}{p}}.$$
	When $p = \infty$, we write $L_{\infty}(\mathcal{M},\tau) := \mathcal{M}$. 
  We refer interested readers to \cite{Hiai, pisier2003non, Terp81} for further details on the theory of non-commutative integration and non-commutative $L_p$-spaces.

Let $G$ be a separable locally compact group endowed with a \textit{left Haar measure} $\mu$. Throughout this paper, all locally compact groups are assumed to be separable, and we shall omit this assumption in the sequel. 
Let $\lambda$ be the \textit{left regular representation}  of $G$ on $L_2(G)$, defined by $\lambda(s)(f)(t) := f(s^{-1}t)$ for every $f\in L_2(G)$ and $s,t\in G$.
It induces a non-degenerate representation of $L_1(G)$ on $L_2(G)$, which we also denote by $\lambda$ without ambiguity and $\lambda: L_1(G)\to B(L_2(G))$ is given by 
	$$\lambda(f)= \int_{G}f(g)\lambda(g)d\mu(g).$$ 
    The \textit{group von Neumann algebra} $\mathcal{L}G\subseteq B(L_2(G))$ is defined as the weak closure of $\lambda(L_1(G))$. 
When $G$ is unimodular, there is a n.f.s. trace on $\mathcal{L}G$ called the \textit{Plancherel trace} (see for instance \cite[Theorem 7.2.2 \& Proposition 7.2.8]{eilers2018c}).
\begin{defn}
	The \textit{Plancherel trace} $\tau: \mathcal{L}G_+\to [0,\infty]$ is defined by 
	$$\tau(\lambda(f)^*\lambda(f)) = 
		\|f\|_2^2 = \int_{G}|{f}(g)|^2d\mu(g),
	$$
 where $ f\in L_2(G)$ is such that the convolution operator $g\mapsto f*g$ defines a bounded operator on $L_2(G)$; and $\tau(\lambda(f)^*\lambda(f))=\infty$ otherwise.
\end{defn}
The above trace satisfies the following \textit{Plancherel identity}:
$$\tau(\lambda (f)) = f(1),$$
where $f\in C_c(G)*C_c(G)$. We denote the non-commutative $L_p$-space associated with $(\mathcal{L}G,\tau)$ 
simply by $L_p(\mathcal{L}G)$.

\subsection{Non-commutative Cotlar identities}
In this subsection, we briefly review the non-commutative Cotlar identities; see \cite{gonzalez2022non-commutative, mei2017free}. 
To this end,  we first recall the notion of conditional expectations in von Neumann algebras. 

	Let $\mathcal{M}$ be a von Neumann algebra equipped with a n.f.s. trace $\tau$, and let $\mathcal{N}$ be a von Neumann subalgebra such that  $\tau_{\mathcal{N}} := \tau|_{\mathcal{N}}$ is still a n.f.s. trace on $\mathcal{N}$. There exists a \textit{conditional expectation} $E: \mathcal{M}\to \mathcal{N}$ which is a linear and positive map satisfying
	\begin{enumerate}
	    \item [(i)] (unital) $E|_{\mathcal{N}} = id_{\mathcal{N}}$;

           \item [(ii)] (projection) $E\circ E =E$;

           \item [(iii)] ($\mathcal{N}$-$\mathcal{N}$-bilinear) $E(y_1xy_2) = y_1E(x)y_2$ for $x\in \mathcal{M}$ and $y_1, y_2\in \mathcal{N}$;
           
  	    \item [(iv)] (trace-preserving) $\tau\circ E = \tau_{\mathcal{N}}$.
	\end{enumerate}
 
The next definition comes from \cite[Definition 1.1]{gonzalez2022non-commutative} or \cite[Proposition 3.2(iv)]{mei2017free}.
\begin{defn}
	Let $\mathcal{M},\,
 \mathcal{N}$ be as above and $E: \mathcal{M}\to \mathcal{N}$ be a conditional expectation. Set $E^{\perp} = \mathrm{id} - E$. A linear operator $T$ on $\mathcal{M}\cap L_2(\mathcal{M})$  satisfies the \textit{non-commutative Cotlar identity relative to} $\mathcal{N}$ if and only if
	\begin{equation}\label{equation; non-commutative Cotlar identity}
		\tag{Cotlar}
		E^{\perp}(T(f)T(f)^*) = E^{\perp}\left[T(fT(f)^*) + T(fT(f)^*)^* - T(T(ff^*)^*)\right],
	\end{equation}
	for every $f\in \mathcal{M}\cap L_2(\mathcal{M})$.
\end{defn}
When considering group von Neumann algebras,  it was shown in \cite[Theorem 1.4]{gonzalez2022non-commutative} that if $T$ is a Fourier multiplier $T_m$ with symbol $m$, the Cotlar identity above admits an alternative formulation with respect to the symbol $m$. Let $G$ be a  locally compact
 unimodular group, and let $G_0$ be an open subgroup of $G$. Set $\mathcal{M} = {\mathcal{L} G}$ and $\mathcal{N}= {\mathcal{L} G_0}$.
Note that the Plancherel trace on ${\mathcal{L} G_0}$ is the restriction of the Plancherel trace on ${\mathcal{L} G}$. The conditional expectation $E: {\mathcal{L} G}\to {\mathcal{L} G_0}$ is given by 
$$E(\lambda (f)) := \int_{G_0}{f}(g)\lambda(g)d\mu(g).$$ 
In this case, \eqref{equation; non-commutative Cotlar identity} holds if and only if \eqref{eq: Cotlar} holds.


 \subsection{Coxeter groups, Coxeter complexes and  Davis complexes} 
 A Coxeter group $$W := \langle S:  (st)^{\mathrm{m}_{st}} = 1\rangle$$
 is called \textit{affine} if all eigenvalues of the matrix $(-2\cos \frac{\pi}{\mathrm{m}_{st}})$ are non-negative; and is \textit{word-hyperbolic} in the sense of Gromov \cite{gromov1987hyperbolic} if there is no subgroup  of $W$ isomorphic to $\mathbb{Z}\times\mathbb{Z}$.
Every Coxeter group is determined by its \textit{Coxeter diagram} that is defined as a graph whose vertex set is $S$ and whose edge set is $\{\{s,t\}: \mathrm{m}_{st}\ge 3\}$. Furthermore, if $\mathrm{m}_{st}>3$, then the edge $\{s,t\}$ is labeled by the integer $\mathrm{m}_{st}$. A Coxeter group is called \textit{irreducible} if its Coxeter diagram is connected.

According to the classification of Coxeter groups, irreducible finite and affine Coxeter groups  have 12 types and 10 types, respectively (see \cite[Section C.1 or Table 6.1]{davis2012geometry}). The following are some typical examples of Coxeter groups. The first example is a type of irreducible finite Coxeter groups, while the second is a type of irreducible affine Coxeter groups.
\begin{ex}[The Coxeter group of type $A_n$]
	Let $S = \{s_i: i=1,\ldots,n\}$. The Coxeter group of type $A_n$ is generated by $S$ subject to the following relations:
	\begin{itemize}
		\item $s_i^2 = 1$ for $1\le i\le n$;
		\item $s_is_{i+1}s_i = s_{i+1}s_is_{i+1}$ for $1\le i \le n-1$;
		\item $s_is_j = s_js_i$ for $|i-j|\ge 2$.
	\end{itemize}
	The Coxeter diagram of type $A_n$ is as follows:
\begin{center}
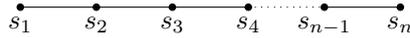

	\begin{tikzpicture}
		\coordinate[label = below: $s_1$] (s) at (0,0);
		\coordinate[label = below: $s_2$] (t) at (1,0);
		\coordinate[label = below: $s_3$] (r) at (2,0);
		\coordinate[label = below: $s_4$] (k) at (3,0); 
		\coordinate[label = below: $s_{n-1}$] (h) at (4,0);
		\coordinate[label = below: $s_n$] (g) at (5,0);
		\node at (s)[circle, fill, inner sep=1pt]{};
		\node at (t)[circle, fill, inner sep=1pt]{};
		\node at (r)[circle, fill, inner sep=1pt]{};
		\node at (k)[circle, fill, inner sep=1pt]{};
		\node at (h)[circle, fill, inner sep=1pt]{};
		\node at (g)[circle, fill, inner sep=1pt]{};

		\draw (s)--(t) node[above, midway]{};
		\draw (t)--(r) node[above, midway]{};
		\draw (r)--(k) node[above, midway]{};
		\draw [dotted] (k)--(h) node[above, midway]{};
		\draw (h)--(g) node[above, midway] (line){};
	\end{tikzpicture}
    \captionof{figure}{The Coxeter diagram of type $A_n$.}
	\end{center}
\end{ex}
It is well known that $A_n = \Sigma_{n+1} (n\ge 1)$, where $\Sigma_{n+1}$ denotes the symmetric group on $n+1$ letters.

\begin{ex}[The Coxeter group of type $\widetilde{A}_n$]\label{def; A2-tilde}
	Let $S = \{s_i:i=1,\ldots ,n+1 \}$. The Coxeter group of type $\widetilde{A}_n$ is generated by $S$ subject to the following relations:
    \begin{itemize}
		\item $s_i^2 = 1$ for $1\le i\le n+1$;
		\item $s_is_{i+1}s_i = s_{i+1}s_is_{i+1}$ for $1\le i \le n$;
		\item $s_is_j = s_js_i$ for $|i-j|\ge 2$ except the case $s_1s_{n+1} = s_{n+1}s_1$;
        \item $s_1s_{n+1}s_1 = s_{n+1}s_1s_{n+1}$.
	\end{itemize}
    The Coxeter diagram of type $\widetilde{A}_n$ is as follows:

    \begin{center}
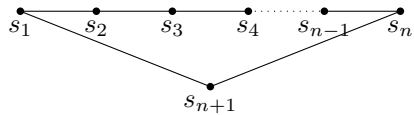

    \begin{tikzpicture}
		\coordinate[label = below: $s_1$] (s) at (0,0);
		\coordinate[label = below: $s_2$] (t) at (1,0);
		\coordinate[label = below: $s_3$] (r) at (2,0);
		\coordinate[label = below: $s_4$] (k) at (3,0); 
            \coordinate[label = below: $s_{n+1}$] (l) at (2.5,-1); 
		\coordinate[label = below: $s_{n-1}$] (h) at (4,0);
		\coordinate[label = below: $s_n$] (g) at (5,0);
		\node at (s)[circle, fill, inner sep=1pt]{};
		\node at (t)[circle, fill, inner sep=1pt]{};
		\node at (r)[circle, fill, inner sep=1pt]{};
		\node at (k)[circle, fill, inner sep=1pt]{};
            \node at (l)[circle, fill, inner sep=1pt]{};
		\node at (h)[circle, fill, inner sep=1pt]{};
		\node at (g)[circle, fill, inner sep=1pt]{};

		\draw (s)--(t) node[above, midway]{};
		\draw (t)--(r) node[above, midway]{};
		\draw (r)--(k) node[above, midway]{};
		\draw [dotted] (k)--(h) node[above, midway]{};
		\draw (h)--(g) node[above, midway] (line){};
            \draw (l)--(g) node[above, midway] (line){};
            \draw (s)--(l) node[above, midway] (line){};
	\end{tikzpicture}
    \captionof{figure}{The Coxeter diagram of type $\widetilde{A}_n$.} 
    \end{center}

\end{ex}

There is another type of Coxeter groups of special interest to us, namely the right-angled Coxeter groups. 

\begin{ex} [Right-angled Coxeter groups]
Let $(W,S)$ be a Coxeter system. We say that $W$ is a \textit{right-angled Coxeter group} if $\mathrm{m}_{st} \in\{2,\infty\}$ for all $s\ne t\in S$.
\end{ex}
Note that every right-angled Coxeter group can be expressed as an amalgamated free product by \cite[Lemma 3.20]{green1990graph}.  Later we will show that one can define Hilbert transforms on a generalized class of right-angled Coxeter groups 
satisfying \eqref{eq: Cotlar}.

Now we recall some properties of reduced expressions of elements in Coxeter groups.
\begin{defn}
	Let $(W,S)$ be a Coxeter system. Denote the \textit{length} of an element $w\in W$ by $l(w)$, which is the minimum number $k$ such that one can write $w= s_1\ldots s_k$ with $s_i\in S$ for $ i\in \{1,\ldots ,k-1\}$. 
    A word $w$ is called a \textit{reduced word} if $w$ can be expressed as $w = s_1\ldots s_{l(w)}$. We call $s_1\ldots s_{l(w)}$ a \textit{reduced expression} of $w$.
\end{defn}
\begin{defn}
	Let $(W,S)$ be a Coxeter system. An \textit{$M$-operation} on a word is one of the following steps:
	\begin{enumerate}
    \item [(a)]Delete a subword of the form $ss$ for any $s\in S$;	
 \item [(b)] Given $s\ne t\in S$ with $\mathrm{m}_{st}<\infty$, replacing a subword of length $\mathrm{m}_{st}$, $st\ldots $ by $ts\ldots $ via the braid relation $(st)^{\mathrm{m}_{st}}=1$.
	\end{enumerate}
	A word is called \textit{$M$-reduced} if it cannot be shortened by finitely many $M$-operations. 
\end{defn}


\begin{prop}(\cite[Theorem 3.4.2]{davis2012geometry})\label{them; M-operation} Let $(W,S)$ be a Coxeter system.
	\begin{enumerate}
    \item [(i)] A word is reduced if and only if it is $M$-reduced.
	\item [(ii)] Let $w\in W$ with $l(w) = n\in \mathbb{N}$.  $s_1\ldots s_n$ and $t_1\ldots t_n$ are reduced expressions of the same word $w$ if and only if one can transform $s_1\ldots s_n$ into $t_1\ldots t_n$ by a sequence of $M$-operations of type $(b)$. 
 \end{enumerate}
\end{prop}

Let us now recall the deletion property and the exchange property of Coxeter groups. In the following proposition, $\hat{s}_i$ indicates that $s_i$ is omitted.

\begin{prop}(\cite[Section 3.2]{davis2012geometry})
Let $(W,S)$ be a Coxeter system and $w = s_1\ldots s_k\in W$ be a reduced word.
\begin{enumerate}
	    \item [(i)]	[Exchange Property] If $s\in S$ and $l(ws)< l(w)$, then $w = s_1\ldots \hat{s}_i\ldots s_ks$ for some $i\in \{1,\ldots,k\}$.
  \item [(ii)] [Deletion Property] If $l(w)<k$, then $w = s_1\ldots\hat{s}_i\ldots\hat{s}_j\ldots s_k$ for some $i\ne j$.
    \end{enumerate}
\end{prop}
The following corollary is an immediate consequence of the exchange property.
\begin{cor}\label{cor; length-starting} 
Let $(W,S)$ be a Coxeter system.
\begin{enumerate}
	    \item [(i)]	For $s\in S$ and $w\in W$, one has $l(sw)< l(w)$ (resp. $l(ws)< l(w)$) if and only if there is a reduced expression of $w$ starting with $s$ (resp. ending with $s$).
	
  \item [(ii)] For $s\in S$ and $w\in W$, one has $l(sw)> l(w)$ (resp. $l(ws)> l(w)$) if and only if there is no reduced expression of $w$ starting with $s$ (resp. ending with $s$).
    \end{enumerate}
\end{cor}

	Coxeter complexes and  Davis complexes are two standard geometric realizations of Coxeter groups, and Coxeter groups also admit natural actions on them. In the following, we give a self-contained introduction to these two types of complexes; see \cite{davis2012geometry,ronan2009lectures,thomas2018geometric} for further details. 

\begin{defn}(\cite[Definition 4.1]{thomas2018geometric})\label{Def: Abstract Simplicial Complex}
Let $V$ be a set.  A \textit{simplicial complex} 
is a collection $\mathcal{K}$ of finite subsets of $V$ (including the empty set), where $V$ is called the \textit{vertex set}, such that 
\begin{enumerate}
    \item[(i)] every singleton set $\{v\}$ is in $\mathcal{K}$;
    \item[(ii)] for every $K\in \mathcal{K}$, all subsets of $K$ are in $\mathcal{K}$. 
\end{enumerate}
Elements of $\mathcal{K}$ are called \textit{simplices}. Every subset of a simplex $K$ is said to be a \textit{face} of $K$. The \textit{dimension} of a simplex $K$ is $|K| - 1$. A $k$-simplex is a simplex of dimension $k$. 
The dimension of $\mathcal{K}$ is the maximal dimension among its simplices. If a subcollection of $\mathcal{K}$ is itself a simplicial complex, it is called a \textit{subcomplex} of $\mathcal{K}$. 
\end{defn}

\begin{defn}
Let $\mathcal{K}$ be a simplicial complex with vertex set $V$. A \textit{simplicial action} of a group $G$ on $\mathcal{K}$ is an action of $G$ on $V$ such that, for every simplex $K \in \mathcal{K}$ and every $g \in G$, the image $g\cdot K$
is again a simplex of $\mathcal{K}$.
\end{defn}


In this paper, we will only consider finite-dimensional simplicial complexes. 
The simplicial complex $\mathcal{K}$ of Definition \ref{Def: Abstract Simplicial Complex} admits a \textit{geometric realization}. For a $k$-simplex $K = \{v_0,\ldots,v_k\}$ in a simplicial complex, its geometric realization $|K|$ is obtained by identifying the vertices of $K$ with the canonical basis vectors of $\mathbb{R}^{k+1}$ and taking their convex hull. 
Then the geometric realization of $\mathcal{K}$ denoted by $|\mathcal{K}|$ is defined by gluing the simplices $|K|$ along common faces.
Throughout the paper, we identify a simplicial complex $\mathcal{K}$ with its geometric realization and equip it with the induced topology.

\begin{rem}
Note that any partially ordered set (poset) $P$ 
gives rise to a simplicial complex $\mathcal{K}(P)$ which we call the \textit{geometric realization of $P$}. The elements of $P$ are vertices of $\mathcal{K}(P)$ and the simplices of $\mathcal{K}(P)$ are the finite totally ordered subsets of $P$. 
 \end{rem}


To properly define Coxeter complexes and Davis complexes, we first introduce a general construction that includes these two types of complexes as special cases.

\begin{defn}
Let $(W, S)$ be a Coxeter system, and $T\subsetneq S$. We call the subgroup $W_T$ generated by $T$ a \textit{parabolic subgroup}. When $T$ is a singleton set, say $\{s\}$, we simplify $W_{\{s\}}=\{e, s\}$ as $W_s$.    
\end{defn}

\begin{defn}[The Basic Construction] (\cite[Definitions 4.2, 4.5]{thomas2018geometric})\label{def; Basic-constraction}
Let $(W, S)$ be a Coxeter system and $X$ be a simplicial complex equipped with a collection of nonempty closed subcomplexes $X_s$ $(s\in S)$.
\begin{enumerate}
    \item [(i)] The space $X$ is called a \textit{mirror space over $S$} and $X_s$ ($s\in S$) are called \textit{mirrors}.

    \item [(ii)] For every simplex $x\in X$, let $S(x) := \{s\in S: x\in X_s\}$. Define the \textit{basic construction} to be the following quotient space:
$$\mathcal{U}(W,X) := (W\times X)/\sim,$$
where $(w,x)\sim (w',y)$ if and only if $x= y$ and $w^{-1}w'$ is in $W_{S(x)}$. Equipped with the quotient topology, $\mathcal{U}(W,X)$ is a connected topological space.  
\end{enumerate}
The basic construction $\mathcal{U}(W,X)$ obtained above forms a simplicial complex. 
The simplices in $\mathcal{U}(W,X)$ are denoted by $[w,x]$.
\end{defn}

\begin{defn}
Let $\mathcal{U}(W,X)$ be a basic construction.
\begin{enumerate}  
\item [(i)] The map $i: X
\rightarrow \mathcal{U}(W,X)$
defined by $x\mapsto [1, x]$ is an embedding. We identify $X$ with its image under $i$ and call it the \textit{fundamental chamber}.
The image of $\{w\}\times X$  for any $w\in W$ in $\mathcal{U}(W,X)$ is denoted by $wX$ and called a \textit{chamber}.

\item [(ii)]  Two chambers $w_1X$ and $w_2X$ are said to be \textit{$s$-adjacent} if $w_1^{-1}w_2 = s\in S$, where $w_1, w_2\in W$. Two chambers are called \textit{adjacent} if they are $s$-adjacent for some generator $s$.

\item [(iii)] Define the $W$-action to be a simplicial action of $W$ on $\mathcal{U}(W,X)$  by $(g,[w,x] )\mapsto[gw, x]$. In particular, this action sends chambers to chambers, i.e. $(g, wX)\mapsto gwX$.
\end{enumerate}
\end{defn}
 Now we are ready to give the definitions of Coxeter complexes and Davis complexes. We say that $K$ is a \textit{full $k$-simplex} if $K$ is the simplicial complex consisting of all subsets of $\{1,2,\ldots,k+1 \}$.

\begin{defn}(Coxeter Complexes) (\cite[Example 4.7, Theorem 5.23]{thomas2018geometric})
Let $(W,S)$ be a Coxeter system. Let $K$ be a full $(|S|-1)$-simplex with mirrors $K_s$ ($s\in S$) being the codimension 1 faces labeled by $s$. The resulting simplicial complex $\mathcal{U}(W,K)$ is called the \textit{Coxeter complex} $\Sigma(W,S)$. Equivalently, the (barycentric subdivision of the) Coxeter complex can be defined as the geometric realization of the poset $\{wW_T: w\in W,\; T\subsetneq S\}$ ordered by inclusion. 
\end{defn}


\begin{defn}(Davis Complexes) (\cite[Definitions 5.6, 5.9, Corollary 5.21, Theorem 5.22]{thomas2018geometric})\label{defn; Davis complex}
Let $(W, S)$ be a Coxeter system. Consider the poset $P=\{T \subseteq S: W_T \text{ is finite}\}$  and let $K=\mathcal{K}(P)$ (the geometric realization of $P$). The mirror $K_s$ ($s\in S$) is the codimension 1 subcomplex where every vertex contains $s$. 
The resulting simplicial complex $\mathcal{U}(W,K)$ is called the \textit{Davis complex} $\Xi(W,S)$. Equivalently, the Davis complex can be defined as the geometric realization of the poset $\{wW_T: w\in W,\; T\subseteq S, W_T \text{ is finite} \}$ ordered by inclusion. 
\end{defn}



\begin{rem} \label{rem;locally finite prop}
Note that the Coxeter complex of $W$ can be embedded in the real vector space where the Tits representation of $W$ acts. 
 Davis complexes can be realized as locally finite subcomplexes of Coxeter complexes (in their poset realization), since for any vertex $wW_T$ in the Davis complex, there are $|W_T|<\infty$ many chambers in the Davis complex containing it.
	Moreover, when Coxeter groups are irreducible and affine, every proper parabolic subgroup $W_T$ is finite for $T\subsetneq S$. In this case, Davis complexes coincide with Coxeter complexes under the above realization. 
\end{rem}

On $\Sigma(W,S)$,  there is a distance function between chambers. For any $w_1, w_2\in W$, the distance between the chambers $w_1K$ and $w_2K$ is defined by
$$d(w_1K, w_2K):= l(w_1^{-1}w_2). $$
Note that $W$ acts freely and transitively on the chambers of $\Sigma(W,S)$ by the $W$-action.
Moreover, the action also preserves the above distance function, that is, $d(gw_1K, gw_2K) = d(w_1K, w_2K)$ for any $g\in W$.
To verify that $d$ is a well-defined distance function, it remains only to check the triangle inequality, since the other two metric axioms are immediate.  This follows from the subadditivity of the length function: $l(xy)\le l(x) +l(y)$ for any $x,y\in W$. One can immediately obtain 
$$d(w_1K, w_2K) = d(w^{-1}w_1K, w^{-1}w_2K)\le d(w_1K, wK)+ d(wK, w_2K),$$ 
for any $w_1, w_2$ and $w$ in $W$.

The Coxeter complex $\Sigma(W,S)$ and the Davis complex $\Xi(W,S)$  possess  rich geometric structures, such as walls and roots, which will be essential for the  construction of  Hilbert transforms on Coxeter groups in the next section.  


\begin{defn}\label{def; roots-in-Cox-complexes}(\cite[Section 2.2]{ronan2009lectures})
Let $(W,S)$ be a Coxeter system. 
\begin{enumerate}
    \item[(i)]  For any $s\in S$ and $w\in W$, the conjugate $wsw^{-1}$ is called a \textit{reflection} of $W$.
Define $\Gamma := \{wsw^{-1}: w\in W, s\in S\}$ to be the set of all reflections of $W$. 
 \item[(ii)] For $r\in \Gamma$, let $H_r$ denote the set of simplices of $\Sigma(W,S)$ fixed by the action of $r$. The subcomplex $H_r$ is called the \emph{wall} labeled by $r$. 
  \item[(iii)] For each wall $H_{r}$, it partitions the chambers of $\Sigma(W,S)$ into two roots (or half-spaces) which are interchanged by $r$.
  We define the \textit{positive root $H_{r}^+$} to be the half containing the fundamental chamber and the \textit{negative root $H_{r}^-$} to be the other half. So we have $\Sigma(W,S)=H_{r}^+ \cup H_{r}^-$ where $H_{r}^+\cap H_{r}^- = H_r$.  
\item[(iv)]Denote the set of all roots of $\Sigma(W,S)$ by $\mathcal{H}$ and the set of all walls by $\partial\mathcal{H}$.
\end{enumerate}
\end{defn}

\begin{rem}\label{rem; unique-wall}
  Note that any two adjacent chambers are separated by a unique wall. Indeed, for any $w\in W$ and $s\in S$, it is obvious that the reflection $wsw^{-1}$ is the unique reflection interchanging the adjacent chambers $wK$ and $wsK$. By \cite[Proposition 2.6(iii)]{ronan2009lectures}, there is a one-to-one correspondence between $\Gamma$ and $\partial\mathcal{H}$ with $r\in \Gamma$ corresponding to the wall labeled by $r$. So the unique wall separating $wK$ and $wsK$ is $H_{wsw^{-1}}$. Moreover, the intersection of $wK$ and $wsK$ in $\Sigma(W,S)$ is a copy of the mirror $K_s$ in these two chambers. Since the unique reflection fixing this intersection is $wsw^{-1}$, the intersection $wK\cap wsK$ is contained in the unique wall $H_{wsw^{-1}}$.
\end{rem}

\begin{rem}\label{rem: expression of half spaces}
By \cite[Proposition 2.6(ii)]{ronan2009lectures},
any root $\alpha\in \mathcal{H}$ of  $\Sigma(W,S)$
can be characterized by a pair of adjacent chambers $x, y$ in $\Sigma(W,S)$ such that $x\in \alpha$ and $y\notin \alpha$, and $\alpha=\{z: d(z,x)<d(z,y)\}$. Let $r\in \Gamma$ and $K$ be the fundamental chamber in $\Sigma(W,S)$. Then there exist $w\in W$ and $ s\in S$ such that $r=wsw^{-1}$.
    When $l(w)< l(ws)$, 
$$H_{wsw^{-1}}^{+} = H(wK,wsK) = \{gK\in\Sigma(W,S): d(gK, wK)< d(gK, wsK)\};$$
when $l(w)> l(ws)$,
$$H_{wsw^{-1}}^{+} = H(wsK, wK) = \{gK\in\Sigma(W,S): d(gK, wK)> d(gK, wsK)\}.$$
In both cases,  the negative root $H_{wsw^{-1}}^{-}$ has the following expression:
$$H_{wsw^{-1}}^{-} = \{gK\in\Sigma(W,S): gK \notin H_{wsw^{-1}}^{+} \}.$$
Moreover, the chambers $wK$ and $wsK$ may be replaced by any pair $w'K$ and $w'tK$, where $w'\in W$ and $t\in S$ satisfy
$wsw^{-1} = w'tw'^{-1}$.
\end{rem}

Note that for any $s\in S$, by Definition \ref{def; roots-in-Cox-complexes}, $K\in H_s^+$ and $s$ interchanges $H_s^+$ and $H_s^-$, so we have $sK\in H_s^-$. It then follows directly from  the above remark that we have the following useful characterization of the positions of chambers in terms of the word length.

\begin{prop}\label{prop; charac-root with length}
		Let $s\in S$ and $g\in W$.
\begin{enumerate}
	    \item [(i)]		$gK\in H_s^+$ (i.e. $d(gK, K) < d(gK, sK)$) if and only if $l(g)< l(sg)$.
		
  \item [(ii)]	$gK\in H_s^-$ (i.e. $d(gK, K) > d(gK, sK)$) if and only if $l(g)> l(sg)$.
  \end{enumerate}
\end{prop}

It is well known that the action of $W$ on $\Sigma(W,S)$ preserves the geometric structures mentioned above, i.e. it maps walls to walls and roots to roots. For the sake of completeness, we include a short proof of the following result. 

\begin{prop} \label{Prop: actions on half spaces}
Let $(W,S)$ be a Coxeter system, $w\in W$ and $s\in S$. Denote the stabilizer of the root ${ H_s^+}$ by ${\rm Stab}_{ H_s^+}$, i.e. ${\rm Stab}_{ H_s^+}= \{g\in W: g\cdot H_s^+ = H_s^+\}$. The following results hold.
\begin{enumerate}
    \item [(i)] 
$w\cdot H_s = H_{wsw^{-1}}$.

\item [(ii)] 
  One has $w\cdot H_s^{+}=H(wK,wsK)$ and $w\cdot H_s^{-}= H(wsK,wK)$. In other words, 
$$w\cdot H_s^{+} = \left\{\begin{array}{l}
			H_{wsw^{-1}}^{+}, \quad \text{if } \; l(w)< l(ws)\\
			H_{wsw^{-1}}^{-}, \quad \text{if } \; l(w)> l(ws). 
		\end{array} \right.$$

\item [(iii)] Set $T_s = \{t\in S: \mathrm{m}_{st} = 2\}$. Then $W_{T_s}\subseteq {\rm Stab}_{ H_s^+}$ and $w\cdot H_s^{+} = H_s^{-}$ if $sw \;\text{or}\;ws\in {\rm Stab}_{ H_s^+}$.


\end{enumerate}
\end{prop}

\begin{proof}
    (i) Note that $w\cdot H_s$ and $H_{wsw^{-1}}$ are walls separating $wK$ and $wsK$. 
    Thus by Remark \ref{rem; unique-wall}, we have $w\cdot H_s = H_{wsw^{-1}}$.

    (ii) 
    The results in (ii) follow directly from Remark \ref{rem: expression of half spaces} and the fact that the action preserves the distance $d$.

    (iii) Note that $s\notin T_s$. It follows directly from part (ii) that $W_{T_s}\subseteq {\rm Stab}_{ H_s^+}$. On the other hand, 
    note that for any $t\in S$, $t\cdot H_s^+ = H_s^-$ if and only if $t = s$. Therefore, we get the conclusion directly.

\end{proof}

	
	

\section{Hilbert transforms on Coxeter groups}\label{section: Cotlar}

Based on the background introduced in the previous section, we now give the proper definition of the multiplier on Coxeter groups. Let $W$ be a Coxeter group and $\Sigma(W,S)$ be its Coxeter complex. For $s\in S$, denote by $H_s$ the wall associated with $s$, and by $H_s^+$ and $H_s^-$ the corresponding positive and negative roots, respectively.
We define
\begin{equation}\label{eq: def of symbol on W}
\tag{MW}
		{m}_W^s(g) := \left\{\begin{array}{l}
			1, \quad \text{if } \; gK\in H_s^{+}\\
			-1, \quad \text{if } \; gK\in H_s^{-}.
		\end{array} \right.
\end{equation}
The above multiplier will give rise to a candidate for Hilbert transforms on $W$, i.e. the Fourier multiplier $T_{{m}_W^s}$. 
To investigate its boundedness on $L_p(\mathcal{L}W)$, the main tool we will use is the Cotlar identity. If ${m}^s_W$ defined above satisfies \eqref{eq: Cotlar} relative to a certain proper subgroup, then the boundedness follows from \ref{prop; Lp-bounded-introduction}. However, this is not always the case unless some additional conditions on the roots in $\Sigma(W,S)$ are imposed. Nevertheless, one can still define Hilbert transforms on general Coxeter groups even though they fail to satisfy the Cotlar identity; see the final subsection of this section for further details.
\subsection{The nested condition}

Now we discuss the nested condition for roots in $\Sigma(W,S)$, which is the key property required for the Coxeter group $W$ to establish the Cotlar identity. Then \ref{theoremA} follows. 
 Consider the roots associated with the wall $H_s$, and let $g\in W$ act on them. We will then obtain two new roots associated with the wall $g\cdot H_s=H_{gsg^{-1}}$. An examination of the multiplier defined in \eqref{eq: def of symbol on W} shows that it satisfies \eqref{eq: Cotlar} relative to a certain proper subgroup provided that, one of the new roots obtained by the action of any element is a subset of an original root.

\begin{defn}\label{Defn;Nested-condition} (Nested Condition)
    Let $(W,S)$ be a Coxeter system, let $L$ be a subset of $W$, and let $s\in S$. We say that $L$ satisfies the
    \textit{nested condition relative to $s$} if, for every $g\in L$, one of the following cases holds: 
    \begin{equation*}\label{nested-condition}
        \tag{NestC}
        g\cdot H_s^+=H_s^+,\;  g\cdot H_s^+= H_s^-,\; g\cdot H_s^+\subsetneq H_s^+,\; g\cdot H_s^+\subsetneq H_s^-,\; g\cdot H_s^-\subsetneq H_s^+,\; \text{or}\; g\cdot H_s^-\subsetneq H_s^-.
    \end{equation*}
\end{defn}

The six cases in \eqref{nested-condition} can be subdivided into two classes, depending on whether $gK$ is in the positive root $H_s^{+}$ or the negative root $H_s^{-}$. We show below that one of these two classes holds if and only if $gK$ is in a root associated with $H_s$.

\begin{prop}\label{prop;nested-condition}
     Let $(W,S)$ be a Coxeter system, let $g\in W$, and let $s\in S$. Suppose that $W$ satisfies the nested condition relative to $s$. 

   \begin{enumerate}
	    \item [(i)] If $gK\in H_s^+$, then one of the following cases holds:
	$$g\cdot H_s^{+} = H_s^{+},\quad g\cdot H_s^{+} \subsetneq H_s^{+},\quad\text{or}\quad g\cdot H_s^{-} \subsetneq H_s^{+}.$$	
	 \item [(ii)] If $gK\in H_s^-$, then one of the following cases holds:
	$$g\cdot H_s^{+} = H_s^{-},\quad g\cdot H_s^{+} \subsetneq H_s^{-},\quad\text{or}\quad g\cdot H_s^{-} \subsetneq H_s^{-}.$$
    \end{enumerate}
\end{prop}
\begin{proof}
    (i) 
    Since $gK\in H_s^+$, it follows from \eqref{nested-condition} that one of the following cases holds:
$$g\cdot H_s^+ = H_s^+,\quad g\cdot H_s^+\subsetneq H_s^+,\quad g\cdot H_s^+\subsetneq H_s^-,\quad g\cdot H_s^-\subsetneq H_s^+,\quad \text{or}\quad g\cdot H_s^-\subsetneq H_s^-.$$ 
Now we claim that neither of the following cases can occur:
$$g\cdot H_s^+ \subsetneq H_s^-\quad \text{or}\quad g\cdot H_s^-\subsetneq H_s^-.$$
If $g\cdot H_s^+ \subsetneq H_s^-$, then $g\cdot H_s^+\cap H_s^+ \subsetneq H_s^-\cap H_s^+ = H_s$, which implies that $g\cdot H_s^+\cap H_s^+$ contains no chamber. However, this contradicts the fact that $gK\in g\cdot H_s^+\cap H_s^+$. 
If $g\cdot H_s^-\subsetneq H_s^-$, 
then $gsK\in  g\cdot H_s^-= H_s^-\cap g\cdot H_s^-$ and $gK\in H_s^+\cap g\cdot H_s^+$. 
Since $gsK$ and $gK$ are adjacent, and both walls $g\cdot H_s$ and $H_s$ separate $gK$ and $gsK$, 
Remark \ref{rem; unique-wall} implies that $g\cdot H_s = H_s$. This contradicts the hypothesis that $g\cdot H_s^-\subsetneq H_s^-$. Thus the case $g\cdot H_s^-\subsetneq H_s^-$ will not happen if $gK\in H_s^+$.

(ii) Since $gK\in H_s^-$, it follows from \eqref{nested-condition} that one of the following cases holds:
$$g\cdot H_s^+ = H_s^-,\quad g\cdot H_s^+\subsetneq H_s^+,\quad g\cdot H_s^+\subsetneq H_s^-,\quad g\cdot H_s^-\subsetneq H_s^+,\quad \text{or}\quad g\cdot H_s^-\subsetneq H_s^-.$$ 
By a similar argument as in (i), one shows that neither of the following cases can occur:
$$g\cdot H_s^+ \subsetneq H_s^+\quad \text{or}\quad g\cdot H_s^-\subsetneq H_s^+.$$

\end{proof}

\begin{rem}\label{rem: Coxeter to Davis}
Note that all the definitions and results established so far in this section carry over to Davis complexes.
\end{rem}


   Now we are ready to prove Theorem A.

\begin{proof}[Proof of \ref{theoremA}]To simplify the notation, we use $m$ to denote $m_W^s$ throughout the proof. 

(i) It suffices to consider the case  $m(g)\ne m(h)$. Suppose first that $m(g) = 1$ and $m(h) = -1$, that is, $gK\in H_s^+$ and $hK\in H_s^-$.
Since $g\notin {\rm Stab}_{ H_s^+}$ and $gK\in H_s^+$,  Proposition \ref{prop;nested-condition}(i) implies that either $g\cdot H_s^{+} \subsetneq  H_s^{+}$ or $g\cdot H_s^{-} \subsetneq  H_s^{+}$.
On the other hand, since $hK\in H_s^-$,  Proposition \ref{prop;nested-condition}(ii) implies that one of the following cases holds:
\begin{equation}\label{equation; geometrical proof-Cotlar-Coxeter}
	h\cdot H_s^{+} = H_s^{-},\quad h\cdot H_s^{-} \subsetneq H_s^{-},\quad\text{or}\quad h\cdot H_s^{+} \subsetneq H_s^{-}.
\end{equation}	
First, we assume that $g\cdot H_s^{+} \subsetneq  H_s^{+}$. Note that $g\cdot H_s^{+} \subsetneq  H_s^{+} \iff H_s^{-} \subsetneq g\cdot H_s^{-}$.
Therefore, if the first case in \eqref{equation; geometrical proof-Cotlar-Coxeter} holds, we have 
$$h\cdot H_s^{+} = H_s^{-}\subsetneq g\cdot H_s^{-},$$
and then
$$g^{-1}h\cdot H_s^{+} = g^{-1}\cdot H_s^{-}\subsetneq H_s^{-}.$$
Applying Proposition \ref{prop;nested-condition}(ii) again, we obtain $$m(g^{-1}h) = m(g^{-1}) = -1.$$
The arguments are similar for the remaining two cases in \eqref{equation; geometrical proof-Cotlar-Coxeter}, where we also get that  for any $h\in W$, $m(g^{-1}h) = m(g^{-1}) = -1$.
Next, assume that $g\cdot H_s^{-} \subsetneq  H_s^{+}$.  Note that $g\cdot H_s^{-} \subsetneq  H_s^{+} \iff H_s^{-} \subsetneq g\cdot  H_s^{+}$.
If the first case  in \eqref{equation; geometrical proof-Cotlar-Coxeter} holds, we have  
$$h\cdot H_s^{+} = H_s^{-}\subsetneq g\cdot H_s^{+},$$
and then
$$g^{-1}h\cdot H_s^{+} = g^{-1}\cdot H_s^{-}\subsetneq H_s^{+}.$$
It follows from Proposition \ref{prop;nested-condition}(i) that $$m(g^{-1}h) = m(g^{-1}) = 1.$$
	The remaining two cases in \eqref{equation; geometrical proof-Cotlar-Coxeter} are treated similarly and also yield  that for any $h\in W$,
$m(g^{-1}h) = m(g^{-1}) = 1$. 
The argument for the case $m(g) = -1$ and $m(h) = 1$ is analogous and is then omitted. Therefore \eqref{eq: Cotlar} is proved.

(ii) We show that $m$ is left $ {\rm Stab}_{ H_s^+}$-invariant. 
For any  $h'\in {\rm Stab}_{ H_s^+}$, we have $h'\cdot H_s^{+} = H_s^{+}$ and $h'\cdot H_s^{-} = H_s^{-}$. Hence, whether $hK\in H_s^{+}$ or $hK\in H_s^{-}$, it follows that $m(h'h)=m(h)$.

Applying \ref{prop; Lp-bounded-introduction}, we conclude that $T_m$ is bounded on $L_p(\mathcal{L}W)$ for every $1< p< \infty$.
\end{proof}


Next, we give a characterization of the multiplier defined in \eqref{eq: def of symbol on W} satisfying \eqref{eq: Cotlar} relative to a subgroup $G_0$. The following result shows that, in constructing
a multiplier by \eqref{eq: def of symbol on W} which satisfies the Cotlar identity with a smallest subgroup (namely, ${\rm Stab}_{ H_s^+}$) removed, the assumption in \ref{theoremA} that $W$ satisfies the nested condition relative to the generator $s$ is optimal.

\begin{prop}\label{prop;equi-char-Cotlar-nested}
    Let $(W,S)$ be a Coxeter system, let $s\in S$, and let $m_W^s$ be given by \eqref{eq: def of symbol on W}. Let $\Gamma_s$ be the maximal subset of $W$ which satisfies the nested condition \eqref{nested-condition} relative to $s$. Then $m_W^s$ satisfies \eqref{eq: Cotlar} relative to a subgroup $G_0$ if and only if the subset $(W \setminus \Gamma_s) \cup {\rm Stab}_{ H_s^+}$ is contained in $G_0$.
\end{prop}

\begin{proof}
In this proof, we simply denote $m_W^s$ by $m$. 
    By \ref{theoremA}, it is clear that if $(W \setminus \Gamma_s) \cup {\rm Stab}_{ H_s^+}\subseteq G_0$, then $m$ satisfies \eqref{eq: Cotlar} relative to $G_0$. 
   Now we prove the reverse implication by showing  that
    for any $g\in (W \setminus \Gamma_s) \cup {\rm Stab}_{ H_s^+}$, 
    there exists $h\in W$ such that 
\begin{equation}\label{eq;7}
    m(h) \ne m(g) \quad \text{and}\quad m(g^{-1}h) \ne m(g^{-1}).
\end{equation}
First note that when $g\in {\rm Stab}_{ H_s^+}$, for any $h\in W$ such that $hK\in H_s^-$, \eqref{eq;7} holds.
Then it remains to consider the case $g\in W \setminus \Gamma_s$. Since $g$ is an arbitrary element in $W$ which fails to satisfy \eqref{nested-condition}, all the following intersections contain at least one chamber:
$$g\cdot H_s^+\cap H_s^+, \quad g\cdot H_s^+\cap H_s^-, \quad g\cdot H_s^-\cap H_s^+, \quad \text{and}\quad  g\cdot H_s^-\cap H_s^-.$$
Without loss of generality, assume that $K, gK\in g\cdot H_s^+ \cap H_s^+$. Consider $h\in W$ with $hK\in g\cdot H_s^-\cap H_s^-$. Then $m(g)\ne m(h)$. However, $g^{-1}hK\in H_s^-\cap g^{-1}\cdot H_s^-$ and $g^{-1}K\in H_s^+$, which implies that $m(g^{-1}h) \ne m(g^{-1})$. Hence
\eqref{eq;7} holds. 
\end{proof}

The nested condition is defined using inclusions between roots in Coxeter complexes, which is not straightforward to verify in practice. 
We will show in the following that the  nested condition \eqref{nested-condition} admits an equivalent algebraic characterization in terms of the braid relations between generators.

\begin{prop}\label{prop;W-satisfies-NestC-equi-to-WRA}
    Let $(W,S)$ be a Coxeter system and let $s\in S$. Let $N_s\subseteq W$ be the set of elements whose reduced expressions neither start nor end with $s$. Then the following statements are equivalent.
    \begin{enumerate}
        \item [(i)] $W$ satisfies the nested condition \eqref{nested-condition} relative to $s$.

        \item [(ii)] For any $u\in S\setminus \{s\}$, $\mathrm{m}_{su}\in \{2,\infty\}$.

    \end{enumerate}
    
\end{prop}
\begin{proof}
    We first show that (i) implies (ii). Suppose that $W$ satisfies the nested condition relative to $s$. Let $u\in S\setminus \{s\}$. Note that $l(u)<l(us)$, which implies $d(uK,K)<d(uK,sK)$. By Remark \ref{rem: expression of half spaces}, 
    this is equivalent to saying $uK\in H_s^+$. Applying Proposition \ref{prop;nested-condition}, one of the following cases holds:
	$$u\cdot H_s^{+} = H_s^{+},\quad u\cdot H_s^{+} \subsetneq H_s^{+},\quad\text{or}\quad u\cdot H_s^{-} \subsetneq H_s^{+}.$$	
    Now we show that $u\cdot H_s^+\subsetneq H_s^+$ cannot occur. Suppose, for contradiction, that $u\cdot H_s^+\subsetneq H_s^+$. Applying the group action of $u$, we have  $H_s^+\subsetneq u\cdot H_s^+$ and then $u\cdot H_s^+\subsetneq u\cdot H_s^+$, which is impossible. 
Now to obtain (ii), it suffices to show that
the following statements hold for any $u\in S\setminus \{s\}$: 
	\begin{enumerate}
	    \item [(a)] $\mathrm{m}_{su} = 2$ if and only if $u\cdot H_s^+ = H_s^+$.
        \item [(b)] If $u\cdot H_s^-\subsetneq H_s^+$, then $\mathrm{m}_{su}= \infty$.
	\end{enumerate}
Statement (a) follows directly from Proposition \ref{Prop: actions on half spaces}. For  (b), assume $\mathrm{m}_{su} \ne \infty$ then there exists $ k \in\mathbb{N}\setminus \{1,2\}$ such that $\mathrm{m}_{su} = k$. Since $(su)^k=1$, the largest word length for elements in $W_{\{s,u\}}$ is $k$. Choose $g\in W_{\{s,u\}}$ such that $l(g) = k$, then $g$ can be written as  $g = usu\ldots = sus\ldots$. Since $g$ admits a reduced expression starting with $s$, by Proposition \ref{prop; charac-root with length} and Corollary \ref{cor; length-starting}, we have $gK\in H_s^{-}$. Then by our assumption $u\cdot H_s^-\subsetneq H_s^+$, we have $ugK\in  H_s^{+}$. This implies that $l(ug)<l(sug)$ by Proposition \ref{prop; charac-root with length}. However, this is not true since $k -1 = l(ug)$ and $l(sug) = k-2$. So we have  $\mathrm{m}_{su}= \infty$. 

\medskip

Next we show that (ii) implies (i).  Let $g\in N_s$. Recall that $T_s=\{u\in S: \mathrm{m}_{su}=2\}$. Then $W_{T_s}\subseteq N_s$. If $g\in W_{T_s}$, 
    then by Proposition \ref{Prop: actions on half spaces},
	$g\cdot H_s^- = H_s^-$. On the other hand, if $g\in N_s\setminus W_{T_s}$, then in any reduced expression of $g$, there is at least one letter $s_\ell\in S$ such that $\mathrm{m}_{ss_{\ell}} =\infty$. We claim that in this case $g\cdot H_s^- \subsetneq H_s^+$. Note that by Proposition \ref{prop; charac-root with length}, Proposition \ref{Prop: actions on half spaces} and Corollary \ref{cor; length-starting}, we have 
    \begin{align*}
      hK\in g\cdot H_s^- &   \iff   d(hK, gK)>d(hK, gsK)\\
      & \iff  l(h^{-1}g)>l(h^{-1}gs)\\
      & \iff \text{There exists a reduced expression of } h^{-1}g \text{ ending with } s,
    \end{align*}
and 	
\begin{align*}
 hK\in H_{s}^+ & \iff    d(hK, K)<d(hK, sK)\\
 & \iff l(h^{-1})<l(h^{-1}s)\\
 &\iff \text{There is no reduced expression of } h^{-1} \text{ ending with } s.
\end{align*}
 Take $hK\in g\cdot H_s^- $, then there exists a reduced expression of $h^{-1}g$ ending with $s$.  Suppose that $hK\notin H_{s}^+ $, then there is a reduced expression of $h^{-1}$ ending with $s$. 
 Since  $g\in N_s$ and there is a letter $s_\ell\in S$ in any reduced expression of $g$ such that $\mathrm{m}_{ss_{\ell}} =\infty$, when multiplying $h^{-1}$ and $g$, the $s$ at the end of the reduced expression of $h^{-1}$ can never reach the end of that of $h^{-1}{g}$ by $M$-operations. This contradicts the assumption that $hK\in g\cdot H_s^- $. Therefore, $hK$ is also in $H_{s}^+ $, which
shows the inclusion $g\cdot H_s^- \subseteq H_s^+$. 
Moreover, it is clear that $gK\in g\cdot H_s^+$. However,  since $g\in N_s$, we have $l(g^{-1})<l(g^{-1}s)$, which shows that $gK\in H_s^+$, so 
$gK\notin H_s^-$. Therefore, if $g\in N_s$, then we have either $g\cdot H_s^{-} = H_s^{-}$ or $g\cdot H_s^{-} \subsetneq H_s^{+}$.

\medskip

 Let $g$ admit a reduced expression ending with $s$ but not starting with $s$, i.e. $g = g_1s$ with $g_1\in N_s\setminus \{1\}$.  If $g_1\in N_s\setminus W_{T_s}$, it follows from the claim in the previous paragraph and Proposition \ref{Prop: actions on half spaces}(iii) that
	$$g\cdot H_s^+ = g_1s\cdot H_s^+ = g_1\cdot H_s^- \subsetneq H_s^{+}.$$
    If $1\ne g_1\in W_{T_s}$, Proposition \ref{Prop: actions on half spaces}(iii) implies that $g\cdot H_s^- = H_s^+$. Next, suppose that $g$ admits a reduced expression starting with $s$ but not ending with $s$, i.e. $g = sg_2$ with $g_2\in N_s\setminus \{1\}$. If $g_2\in N_s\setminus W_{T_s}$, then the claim in the previous paragraph and Proposition \ref{Prop: actions on half spaces}(iii) imply that
	$$g\cdot H_s^- = sg_2\cdot H_s^- \subsetneq s\cdot H_s^+ = H_s^-.$$
    If $1\ne g_2\in W_{T_s}$, then $g = sg_2$ and $g\cdot H_s^- = H_s^+$. 
	Finally, let $g$ admit a reduced expression starting with $s$ and ending with $s$, i.e. either $g = sg_3s$ with $g_3\in N_s\setminus W_{T_s}$ or $g = s$. Similarly, if $g= sg_3s$ with $g_3\in N_s\setminus W_{T_s}$, we have 
	$$g\cdot H_s^+ = sg_3s\cdot H_s^+ = sg_3\cdot H_s^- \subsetneq s\cdot H_s^+ = H_s^-.$$
    If $g= s$, by  Proposition \ref{Prop: actions on half spaces}(iii), $g\cdot H_s^+ = H_s^-$. 
    Therefore, $W$ satisfies the nested condition relative to $s$.
\end{proof}

\begin{cor}\label{Cor;RA-Coxeter gp-NestC}
 Right-angled Coxeter groups satisfy the  nested condition \eqref{nested-condition} relative to all generators. 
\end{cor}

\begin{rem}
    It is easy to see from the proof of (ii) $\Rightarrow$ (i) above that the inclusion relations for $g\notin N_s$ can be deduced from that for $g\in N_s$. Therefore,  to check whether $W$ satisfies the  nested condition \eqref{nested-condition} relative to $s$, it is sufficient to check that for $g\in N_s$, we have either $g\cdot H_s^- = H_s^-$ or $g\cdot H_s^-\subsetneq H_s^+$.
\end{rem}

As a consequence of Proposition \ref{prop;W-satisfies-NestC-equi-to-WRA}, the six cases in \eqref{nested-condition} can be reformulated in terms of the algebraic expressions of $g$ in a subset of $W$. These expressions will be useful in the study of buildings in Section \ref{section;building}.

\begin{prop}\label{prop; equi-char-Nest-algebraic-forms}
    Let $(W,S)$ be a Coxeter system, let $s\in S$, and let $N_s\subseteq W$ be the set of elements whose reduced expressions neither start nor end with $s$. Assume that $W$ satisfies the nested condition relative to $s$. Then for any $g\in W$, the six cases in \eqref{nested-condition} can be reformulated as follows:
    $$g \in W_{T_s},\quad g \in W_{T_s}s,\quad g \in ( N_s\setminus W_{T_s})s,\quad g \in s( N_s\setminus W_{T_s})s,\quad g\in N_s\setminus W_{T_s}\quad \text{or}\quad g \in s( N_s\setminus W_{T_s}).$$
\end{prop}
\begin{proof}
    By the proof of Proposition \ref{prop;W-satisfies-NestC-equi-to-WRA}, the six algebraic expressions in the statement of the proposition imply the six cases in  \eqref{nested-condition}. Thus it suffices to prove the converse. Recall that $T_s=\{t\in S: \mathrm{m}_{st}=2\}$.
    First, we show that $g\cdot H_s^+ = H_s^+$ implies $g\in W_{T_s}\subseteq N_s$. Since $g\in {\rm Stab}_{ H_s^+} = \{w\in W: w\cdot H_s^+ = H_s^+\}$, Proposition \ref{Prop: actions on half spaces} implies that $gs = sg$. Let $s_1\ldots s_k$ be a reduced expression of $g$. By Proposition \ref{prop;W-satisfies-NestC-equi-to-WRA}, we have $\mathrm{m}_{ss_j}\in \{1,2,\infty\}$ for all $j =1,\ldots,k$.  The commutation relation $gs=sg$ forces $\mathrm{m}_{ss_j}=2$ whenever $s_j\neq s$.  
    If there is some $i\in \{1,\ldots,k\}$ such that $s_i =s$, then $g = g's$ for some $g'\in W_{T_s}$.  By Proposition  \ref{Prop: actions on half spaces}(iii), this would imply $g\cdot H_s^+ = H_s^-$, contradicting our assumption $g\in {\rm Stab}_{ H_s^+}$. Thus $g\in W_{T_s}$. A similar argument shows that if $g\cdot H_s^+=H_s^-$, then $g=g_1s$ for some $g_1\in W_{T_s}$.

    Next, we show that $g\cdot H_s^-\subsetneq H_s^+$ implies $g\in N_s\setminus W_{T_s}$. Note that $g\cdot H_s^-\subsetneq H_s^+$ implies that $gK$ and $gsK$ lie in $H_s^+$. Otherwise, $H_s$ would separate $gK$ and $gsK$. The uniqueness of walls separating adjacent chambers (see Remark \ref{rem; unique-wall}) ensures that $g\cdot H_s = H_s$, contradicting the assumption $g\cdot H_s^-\subsetneq H_s^+$. It then follows from $g\cdot H_s^-\subsetneq H_s^+$, i.e. $g^{-1}\cdot H_s^-\subsetneq H_s^+$ that $gK\in H_s^+$ and $g^{-1}K\in H_s^+$. By Proposition \ref{prop; charac-root with length}  and Corollary \ref{cor; length-starting}, neither $g$ nor $g^{-1}$ admits a reduced expression starting with $s$, and then $g\in N_s$. This together with the discussion of the first case shows that $g\in N_s\setminus W_{T_s}$. 

    For the remaining three cases, the arguments are similar so we only give the proof of the implication that  $g\cdot H_s^+ \subsetneq H_s^+$ implies $g = g_2s$ for some $g_2\in N_s\setminus W_{T_s}$. Since $g\cdot H_s^+\subsetneq H_s^+$, i.e. $g^{-1}\cdot H_s^-\subsetneq H_s^-$, we have $gK\in H_s^+$ and $g^{-1}K\in H_s^-$. By Proposition \ref{prop; charac-root with length}  and Corollary \ref{cor; length-starting}, $g$ admits a reduced expression ending with $s$ but not starting with $s$. Therefore $g = g_2s$ for some $g_2\in N_s\setminus W_{T_s}$.
\end{proof}
\begin{rem}
  By Proposition \ref{prop; charac-root with length} and  Corollary \ref{cor; length-starting}, we can easily see that the multiplier defined in \eqref{eq: def of symbol on W} admits an equivalent algebraic description in terms of the length function on $W$:
\begin{equation}\label{eq;algebric-def-symbol-w RA-Cg}
	m_W^s(g) := \left\{\begin{array}{l}
		1, \quad \text{if } \; l(sg) > l(g)\\
		-1, \quad \text{if } \; l(sg) < l(g).
	\end{array} \right.
	\end{equation}  
Equivalently, 
\begin{equation}\label{eq;algebric-def-symbol-w}
	m_W^s(g) := \left\{\begin{array}{l}
		1, \quad \text{if there is no reduced expression of }g \text{ starting with }s\\
		-1, \quad \text{if there is a reduced expression of }g \text{ starting with }s.
	\end{array} \right.
	\end{equation}
In view of the algebraic characterization of the nested condition given in Proposition \ref{prop; equi-char-Nest-algebraic-forms},   \ref{theoremA} admits a purely algebraic proof using only reduced expressions in Coxeter groups. We leave the details to the interested reader.
\end{rem}

\subsection{Hilbert transforms on general Coxeter groups}

In this subsection, our main goal is to prove \ref{theoremB}. To this end, 
 we will show that every finitely generated Coxeter group admits a torsion-free subgroup of finite index that satisfies the nested condition \eqref{nested-condition} relative to a certain generator. Therefore, 
  we can define multipliers on such finite-index subgroups that satisfy \eqref{eq: Cotlar} and then extend them to the whole group. To connect the torsion-freeness with the nested condition for roots, we will make use of the locally finite property of Davis complexes. In the remainder of this subsection, we consider the nested condition for roots in Davis complexes, viewed as subcomplexes of Coxeter complexes
  (see Remark \ref{rem;locally finite prop}).
  

Every Coxeter group $W$ admits a torsion-free subgroup of finite index (see \cite[Corollary D.1.4]{davis2012geometry}), which we denote by $W_0$.  It was shown in \cite{davis2012geometry} (originally due to Millson \cite{Mil76}) that any torsion-free subgroup $W_0$ such that $sW_0s = W_0$ for a generator $s\in S$ will satisfy a \textit{trivial intersection property} in the Davis complex of $W$. We state the corresponding result below. Note that the proof of the following lemma relies essentially on the locally finite property of Davis complexes.  

\begin{lem}(\cite[Lemma 14.1.8]{davis2012geometry})\label{lem;nested-properties for torsion-free normal subg}
 Let  $(W,S)$ be a Coxeter system, and let $W_0\subseteq W$ be a torsion-free subgroup. Consider the Davis complex $\Xi(W,S)$ of $W$ and the walls in it. If there is $s\in S$ such that $sW_0s = W_0$, then the pair $(W_0,s)$ satisfies the following trivial intersection property: 
	$$g\cdot H_s = H_s\quad \text{or}\quad g\cdot H_s\cap H_s = \emptyset,$$
	for every $g\in W_0$. In particular, if $W_0$ is a normal torsion-free subgroup of $W$, then for every $u\in S$, $(W_0,u)$ satisfies the trivial intersection property.
\end{lem}

Although the nested condition in  Definition \ref{Defn;Nested-condition} is formulated in terms of Coxeter complexes, it can be adapted to Davis complexes by replacing walls and roots in Coxeter complexes with their counterparts in Davis complexes. 
Note that chambers in the Davis complex $\Xi(W,S)$ are subcomplexes and are not necessarily simplices. To maintain consistency with the notation for Coxeter complexes, we shall abuse notation slightly and write  $gK\in \Xi(W,S)$ instead of $gK\subseteq \Xi(W,S)$ for every $g\in W$.
Now we define a multiplier $m_{W_0}^s$ on $W_0$, analogous to \eqref{eq: def of symbol on W}, in terms of the Davis complex. Fix $s\in S$. For every $g\in W_0$, define
	\begin{equation}\label{eq: multiplier on W0}
		m_{W_0}^s(g) := \left\{\begin{array}{l}
			1, \quad \text{if} \; gK\in H_s^+\\
			-1, \quad \text{if} \; gK\in H_s^-,
		\end{array} \right.
	\end{equation}
    where $H_s^\pm$ are roots in the Davis complex of $W$.

    \begin{lem}\label{cor;torsion-free-cotlar}
        Let $W_0\subseteq W$ be a torsion-free subgroup of $W$, and let $s\in S$ satisfy $sW_0s = W_0$. Then the multiplier $m^s_{W_0}$ defined in \eqref{eq: multiplier on W0}  satisfies  \eqref{eq: Cotlar} relative to the subgroup $W_0\cap {\rm{Stab}}_{H_s^+}$. 
    \end{lem}
    
\begin{proof}
It is easy to see that the trivial intersection property for the pair $(W_0,s)$ implies that $W_0$ satisfies the nested condition \eqref{nested-condition} relative to $s$ with respect to roots in the Davis complex of $W$. Note that by an argument analogous to that in the proof of \ref{theoremA}, if we only assume a subgroup of $W$ satisfies the nested condition \eqref{nested-condition}, then \eqref{eq: Cotlar} holds relative to the intersection of that subgroup with $\mathrm{Stab}_{H_s^+}$.
The result therefore follows from Lemma \ref{lem;nested-properties for torsion-free normal subg}.
\end{proof}

Suppose that the index of $W_0$ in $W$ is $k\in \mathbb{N}$ and fix a set of representatives $\{w_1,\ldots, w_k\}$ for the left cosets of $W_0$ in $W$, so
that $W= \mathop{\bigsqcup}\limits_{i=1}^k w_iW_0$.
For any $g= w_ih$  with $h\in W_0$, we define an extension of $m_{W_0}^s$ to $W$ by
\begin{equation}\label{eq: def of general symbol}\tag{MU}
m(g)=m_{W_0}^s(h), 
\end{equation}
which is a well-defined multiplier on $W$, depending in general on the chosen system of left-coset representatives.
The Fourier multiplier $T_m$ associated with the multiplier $m$ in \eqref{eq: def of general symbol} gives rise to a Hilbert transform on any finitely generated Coxeter group $W$. 
Now we are ready to prove \ref{theoremB}, which asserts that $T_m$ is bounded on $L_p(\mathcal{L}W)$ for every $1<p<\infty$.

\begin{proof}[Proof of \ref{theoremB}]
 We may  assume without loss of generality that there is $s\in S$ such that $sW_0s = W_0$, since otherwise we may replace $W_0$ by the largest normal subgroup $N(W_0) := \mathop{\cap}\limits_{{g\in W}}gW_0g^{-1}$  contained in $W_0$ which is still torsion-free and of finite index. By an argument similar to that in the proof of \ref{theoremA}, the multiplier  $m_{W_0}^s$ on $W_0$ is left $W_0\cap {\rm{Stab}}_{H_s^+}$-invariant. 
   Moreover, Lemma \ref{cor;torsion-free-cotlar} shows that $m_{W_0}^s$ on $W_0$ satisfies the Cotlar identity, and then $T_{m_{W_0}^s}$ is bounded on $L_p(\mathcal{L}W_0)$ for $1<p<\infty$ by \ref{prop; Lp-bounded-introduction}. Since $W_0$ has finite index in $W$, by a standard argument, it is straightforward to deduce the boundedness of $T_m$  on $L_p(\mathcal{L}W)$ from that of $T_{m_{W_0}^s}$ on $L_p(\mathcal{L}W_0)$.
\end{proof}



\subsection{Examples and counterexamples}

In this subsection, we first present some examples of Coxeter groups satisfying the nested condition and explain the geometric model of the Hilbert transforms defined by \eqref{eq: def of symbol on W} through their Coxeter complexes. We then give some examples of Coxeter groups for which the multiplier defined by \eqref{eq: def of symbol on W} fails to satisfy \eqref{eq: Cotlar}, and construct finite-index subgroups satisfying the nested condition. These examples of groups satisfy property (F$\mathbb{R}$). Nevertheless, by \ref{theoremB}, they still admit $L_p$-bounded Hilbert transforms.
\begin{ex}\label{ex; D-infty}
	All right-angled Coxeter groups satisfy the nested condition relative to every generator by Corollary \ref{Cor;RA-Coxeter gp-NestC}.

\end{ex}

\begin{ex}\label{ex; PGL-2-Z-first}
	Let  $\mathrm{PGL}_2(\mathbb{Z})$ be the group which is the quotient of the group of invertible $2\times 2$ matrices over $\mathbb{Z}$ by identifying each matrix with its negative. By \cite[Example 2.2.3]{abramenko2008buildings}, $\mathrm{PGL}_2(\mathbb{Z})$ is a Coxeter group with three generators 
	$$u = \left(\begin{array}{cc}
		0 & 1\\
		1 & 0\\\end{array} \right),\quad  
		t = \left(\begin{array}{cc}
		-1 & 1\\
		0 & 1\\\end{array} \right)\quad \text{and}\quad 
		s = \left(\begin{array}{cc}
		-1 & 0\\
		0 & 1\\\end{array} \right) $$
such that $\mathrm{m}_{ut}=3$, $\mathrm{m}_{su}=2$ and $\mathrm{m}_{ts}=\infty$. Therefore, $\mathrm{PGL}_2(\mathbb{Z})$ is not a right-angled Coxeter group but it satisfies the nested condition relative to $s$ by Proposition \ref{prop;W-satisfies-NestC-equi-to-WRA}. 
    
    Considering the Tits representation of $\mathrm{PGL}_2(\mathbb{Z})$, the three generators correspond to reflections on the Poincar\'e upper half-plane $\mathbb{H}$ given by $\rho(u): z\mapsto \frac{1}{\bar{z}}$, $\rho(t): z\mapsto 1- \bar{z}$ and $\rho(s): z\mapsto -\bar{z}$.  
They fix three walls $H_u = \{z\in \mathbb{H}: |z| = 1\}$, $H_t = \{z\in \mathbb{H}: {\rm Re}(z) = 1/2\}$ and $H_s = \{z\in \mathbb{H}: {\rm Re}(z) = 0\}$, respectively. Consider the domain 
$$K:= \{z\in \mathbb{H}: 0\le {\rm Re}(z)\le \frac{1}{2}\; \text{and}\; |z|\ge 1\},$$
which is a 2-simplex (a geodesic triangle with one vertex at $\infty$) with mirror structure $\{K_u, K_t, K_s\}$, where $K_u, K_t$ and $K_s$ are the three sides of $K$ lying in $H_u$, $H_t$ and $H_s$, respectively. The Coxeter complex of $\mathrm{PGL}_2(\mathbb{Z})$ is the basic construction $\mathcal{U}(\mathrm{PGL}_2(\mathbb{Z}), K)$, which gives a tessellation of the upper half-plane $\mathbb{H}$ by geodesic triangles. 
A portion of the Coxeter complex is shown in Figure~\ref{figure: GL2}, where the orbit of $K_s$ under the action of $\mathrm{PGL}_2(\mathbb{Z})$ is represented by orange curves, that of $K_u$ by black curves, and that of $K_t$ by teal curves. Observe that the walls $g\cdot H_s$ for any $g\in \mathrm{PGL}_2(\mathbb{Z})$ intersect only at infinity. This reflects the fact that $\mathrm{PGL}_2(\mathbb{Z})$ satisfies the nested condition relative to $s$.
	
The wall $H_s$ divides the Coxeter complex into two parts. Since the fundamental chamber $K$ lies in the root $\{z\in \mathbb{H}: {\rm Re}(z)\geq 0\}$, it is the positive root $H_s^+$. For the multiplier $m$ defined in \eqref{eq: def of symbol on W}, we have  $m(g) = 1$ if  the chamber $gK$ is in the root on the right-hand side of $H_s$, or equivalently, if $g$ does not admit a reduced expression starting with $s$, whereas $m(g) = -1$ if the chamber $gK$ is in the root on the left-hand side of $H_s$, or equivalently, if $g$ admits a reduced expression starting with $s$.

    \begin{center}
	\begin{tikzpicture}
	   \draw [orange] (-1/6,3) coordinate (l3) node {$H_s$};
		\draw [teal] (5/6,3) coordinate (l2) node {$H_t$};
		\draw [black] (1/2,21/10) coordinate (l1) node {$H_u$};
		\draw [black] (1/2,5/2) coordinate (K) node {$K$};
		\draw [black] (-1/2,5/2) coordinate (K) node {$sK$};
		\draw [black] (-3/2,5/2) coordinate (K) node {$stK$};
		\draw [black] (3/2,5/2) coordinate (K) node {$tK$};
		\draw [black] (1/3,{sqrt(3)}) coordinate (K) node {$uK$};
		\draw [black] (-1/3,{sqrt(3)}) coordinate (K) node {$suK$};
		\draw [black] (-5/3,{sqrt(3)}) coordinate (K) node {$stuK$};
		\draw [black] (2/3,6/5) coordinate (K) node {$utK$};
		\draw [black] (4/3,6/5) coordinate (K) node {$utuK$};
		\draw [black] (5/3,{sqrt(3)}) coordinate (K) node {$tuK$};
		\draw [black] (0,0) coordinate  ($0$) node [below] {$0$};
		\draw [black] (1,0) coordinate  ($1/2$) node [below] {$1/2$};
		\draw [black] (-1,0) coordinate  ($-1/2$) node [below] {$-1/2$};
		\draw [black] (2,0) coordinate  ($1$) node [below] {$1$};
		\draw [black] (-2,0) coordinate  ($-1$) node [below] {$-1$};

		\draw (-4,0)--(4,0) {};
		\draw [orange] (0,0)--(0,4);
		\draw [teal] (1,0)--(1,{sqrt(3)/3});
		\draw [black] (1,{sqrt(3)/3})--(1,{sqrt(3)});
		\draw [teal] (1,{sqrt(3)})--(1,4);
		\draw [orange] (2,0)--(2,4);
		\draw [teal] (-1,0)--(-1,{sqrt(3)/3});
		\draw [black] (-1,{sqrt(3)/3})--(-1,{sqrt(3)});
		\draw [teal] (-1,{sqrt(3)})--(-1,4);
		
		\draw [teal] (0,0) arc (180:120:2);
		\draw [teal] (0,0) arc (0:60:2);
		\draw [black] (1,{sqrt(3)}) arc (120:60:2);
		\draw [black] (-1,{sqrt(3)}) arc (60:120:2);

		\draw [orange] (-2,0)--(-2,4);
		\draw [orange] (2,0) arc (0:180:1);
		\draw [orange] (0,0) arc (0:180:1);
		\draw [orange] (1,0) arc (0:180:1/2);
		\draw [orange] (2,0) arc (0:180:1/2);
		\draw [orange] (0,0) arc (0:180:1/2);
		\draw [orange] (-1,0) arc (0:180:1/2);

		\draw [teal] (2,0) arc (0:60:2);
		\draw [black] (1,{sqrt(3)}) arc (60:120:2);
		\draw [teal] (-2,0) arc (180:120:2);
		\draw [teal] (1,{sqrt(3)/3}) arc (60:180:2/3);
		\draw [teal] (-1,{sqrt(3)/3}) arc (60:180:2/3);
		\draw [black] (1,{sqrt(3)/3}) arc (60:30:2/3);
		\draw [black] (-1,{sqrt(3)/3}) arc (60:30:2/3);
		\draw [teal] (1,{sqrt(3)/3}) arc (120:0:2/3);
		\draw [teal] (-1,{sqrt(3)/3}) arc (120:0:2/3);
		\draw [black] (1,{sqrt(3)/3}) arc (120:150:2/3);
		\draw [black] (-1,{sqrt(3)/3}) arc (120:150:2/3);	 
	\end{tikzpicture} 
   
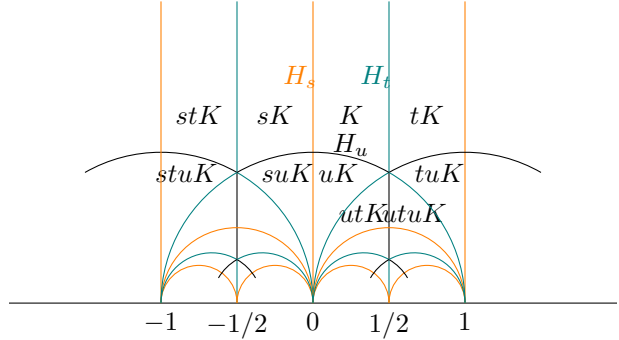
\captionof{figure}{Part of the Coxeter complex of $\mathrm{PGL}_2(\mathbb{Z})$.} 
   \label{figure: GL2}
	\end{center}

Some readers  may be more familiar with the  subgroup $\mathrm{PSL}_2(\mathbb{Z})$ of  $\mathrm{PGL}_2(\mathbb{Z})$, which has index 2 and admits the presentation $$\mathrm{PSL}_2(\mathbb{Z}) = \langle  s_1, s_2: s_1^2 = s_2^3 = 1 \rangle,$$
    where 
    $$s_1 = \left(\begin{array}{cc}
		0 & -1\\
		1 & 0\\\end{array} \right)\quad \text{and}\quad 
		s_2 = \left(\begin{array}{cc}
		0 & -1\\
		1 & 1\\\end{array} \right).$$
        Note that restricting the multiplier $m$ on $\mathrm{PGL}_2(\mathbb{Z})$ defined above to $\mathrm{PSL}_2(\mathbb{Z})$ yields a multiplier on $\mathrm{PSL}_2(\mathbb{Z})$ which coincides with the one studied in \cite[Section 5]{gonzalez2022non-commutative}.

\end{ex}

Next, we show that affine Coxeter groups  do not satisfy \eqref{eq: Cotlar} relative to any proper subgroup. We illustrate this by considering the special case of the Coxeter group $\widetilde{A}_2$, although a similar argument applies to other affine Coxeter groups as well. 
\begin{ex}\label{ex; A-2-tilda-not-satisfying}
	The affine Coxeter group $\widetilde{A}_2$ is defined by 
	$$\widetilde{A}_2 := \langle s, t, u : s^2 = t^2 = u^2 = (st)^3 = (su)^3 = (tu)^3 = 1\rangle,$$
	which is also known as the triangle group $D(3,3,3)$. 
    Consider $K$ to be an equilateral triangle in $\mathbb{R}^2$ with three edges $\{K_s, K_t, K_u\}$ lying in the three walls, denoted by $H_s, H_t$ and $H_u$, respectively. 
    The Coxeter complex is the basic construction $\mathcal{U}(\widetilde{A}_2, K)$, which gives a tessellation of $\mathbb{R}^2$ by equilateral triangles. A portion of the Coxeter complex is shown in Figure \ref{figure: A_2}. The orbit of $K_s $ under the action of $\widetilde{A}_2$ is represented by orange segments, that of $K_u$ by black segments, and that of $K_t$ by teal segments. 
    
	\begin{center}
		
	\begin{tikzpicture}c
	\draw [orange] (0,0)--(1,0);
	\draw [teal] (1/2,{sqrt(3)/2})--(1,0);
	\draw [black] (0,0)--(1/2,{sqrt(3)/2});
	\draw [teal] (-1/2,{sqrt(3)/2})--(1/2,{sqrt(3)/2});
	\draw [orange] (0,0)--(-1/2,{sqrt(3)/2});
	\draw [black] (1/2,{sqrt(3)/2})--(3/2,{sqrt(3)/2});
	\draw [orange] (1,0)--(3/2,{sqrt(3)/2});
	\draw [orange] (1,{sqrt(3)})--(3/2,{sqrt(3)/2});
	\draw [orange] (1,{sqrt(3)})--(0,{sqrt(3)});
	\draw [black] (0,{sqrt(3)})--(1/2,{sqrt(3)/2});
	\draw [orange] (0,{sqrt(3)})--(-1/2,{sqrt(3)/2});
	\draw [black] (0,{sqrt(3)})--(-1,{sqrt(3)});
	\draw [black] (-3/2,{sqrt(3)/2})--(-1,{sqrt(3)});
	\draw [black] (2,0)--(3/2,{sqrt(3)/2});
	\draw [black] (3/2,{-sqrt(3)/2})--(2,0);
	\draw [black] (0,0)--(1/2,{-sqrt(3)/2});
	\draw [teal] (1/2,{-sqrt(3)/2})--(1,0);
	\draw [orange] (3/2,{-sqrt(3)/2})--(1,0);
	\draw [black] (3/2,{-sqrt(3)/2})--(1/2,{-sqrt(3)/2});
	\draw [black] (3/2,{-sqrt(3)/2})--(1,{-sqrt(3)});
	\draw [teal] (-1/2,{-sqrt(3)/2})--(1/2,{-sqrt(3)/2});
	\draw [orange] (0,0)--(-1/2,{-sqrt(3)/2});
	\draw [black] (0,0)--(-1,0);
	\draw [teal] (-1/2,{sqrt(3)/2})--(-1,0);
	\draw [teal] (-1/2,{-sqrt(3)/2})--(-1,0);
	\draw [black] (-3/2,{sqrt(3)/2})--(-1,0);
	\draw [orange] (-3/2,{sqrt(3)/2})--(-1/2,{sqrt(3)/2});
	\draw [orange] (-3/2,{sqrt(3)/2})--(-2,0);
	\draw [orange] (-3/2,{-sqrt(3)/2})--(-2,0);
	\draw [orange] (0,{-sqrt(3)})--(-1/2,{-sqrt(3)/2});
	\draw [black] (0,{-sqrt(3)})--(1/2,{-sqrt(3)/2});
	\draw [orange] (0,{-sqrt(3)})--(1,{-sqrt(3)});
	\draw [black] (0,{-sqrt(3)})--(-1,{-sqrt(3)});
	\draw [black] (-3/2,{-sqrt(3)/2})--(-1,0);
	\draw [orange] (-3/2,{-sqrt(3)/2})--(-1/2,{-sqrt(3)/2});
	\draw [black] (-3/2,{-sqrt(3)/2})--(-1,{-sqrt(3)});

	\draw [teal] (-1,{sqrt(3)})--(-1/2,{sqrt(3)/2});
	\draw [black] (-3/2,{sqrt(27)/2})--(-1,{sqrt(3)});
	\draw [teal]
	(1,{-sqrt(3)})--(1/2,{-sqrt(3)/2});
	\draw [orange] (3/2,{-sqrt(27)/2})--(1,{-sqrt(3)});
	\draw [teal] (2,0)--(1,0);
	\draw [black] (2,0)--(3,0);
	\draw [teal] (-2,0)--(-1,0);
	\draw [orange] (-2,0)--(-3,0);
	\draw [teal] (1/2,{sqrt(3)/2})--(1, {sqrt(3)});
	\draw [orange] (3/2,{sqrt(27)/2})--(1, {sqrt(3)});
	\draw [teal] (-1/2,{-sqrt(3)/2})--(-1, {-sqrt(3)});
	\draw [black] (-3/2,{-sqrt(27)/2})--(-1, {-sqrt(3)});

	\draw [black] (1/2,1/3) coordinate (K) node {$K$};
	\draw [black] (0,1/2) coordinate (uK) node {$uK$};
	\draw [black] (1,1/2) coordinate (tK) node {$tK$};
	\draw [black] (1/2,-1/3) coordinate (tK) node {$sK$};
	\draw [black] (0,-2/3) coordinate (usK) node {$suK$};
	\draw [black] (-1/2,1/5) coordinate (tK) node {$usK$};
	\draw [black] (-1/2,-1/5) coordinate (tK) node {$usuK$};
	\draw [black] (-1,2/3) coordinate (tK) node {$ustK$};
	\draw [black] (0,-21/20) coordinate (tK) node {$sutK$};
	\draw [black] (0,21/20) coordinate (tK) node {$utK$};
	\draw [black] (1,21/20) coordinate (tK) node {$tuK$};
	\draw [black] (1/2,3/2) coordinate (tK) node {$tutK$};
	\draw [black] (3/2,1/5) coordinate (K) node {$tsK$};
	\draw [black] (3/2,-1/5) coordinate (tK) node {$tstK$};
	\draw [black] (1,-2/3) coordinate (utK) node {$stK$};
 \draw [orange] (3,1/6) coordinate (l1) node {$H_s$};
 \draw [black] (3/2,{sqrt(675)/12}) coordinate (l1) node {$H_u$};

	\end{tikzpicture} 
    
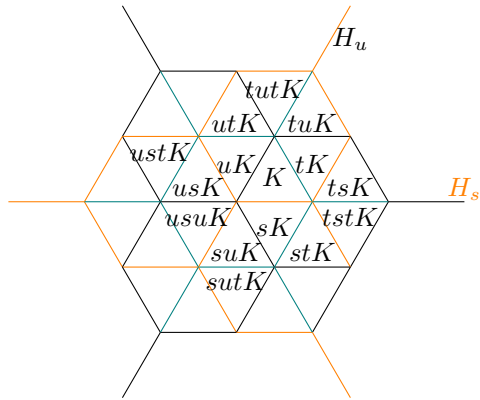
\captionof{figure}{Part of the Coxeter complex of $\widetilde{A}_2$.}
    \label{figure: A_2}
	 \end{center}
Note that in this case, the Davis complex of $\widetilde{A}_2$ coincides with its Coxeter complex. By Proposition \ref{prop;W-satisfies-NestC-equi-to-WRA}, $\widetilde{A}_2$ fails to satisfy the nested condition relative to any generator. This fact can also be observed directly from Figure \ref{figure: A_2}. 
Without loss of generality, we consider the generator $s$.
It is immediate to see from Figure \ref{figure: A_2} that $\{u,t\}$ does not satisfy the nested condition relative to $s$, since both $u\cdot H_s=H_{usu}$ (the line containing the orange edge of $uK$) and $t\cdot H_s = H_{tst}$ (the line containing the orange edge of $tK$) intersect with $H_s$. 
We now show that $m$ defined in \eqref{eq: def of symbol on W} associated with the generator $s$ does not satisfy \eqref{eq: Cotlar} relative to any proper subgroup. By Proposition \ref{prop;equi-char-Cotlar-nested}, it suffices to show that there is no proper subgroup containing $(\widetilde{A}_2 \setminus \Gamma_s) \cup {\rm Stab}_{ H_s^+}$, where $\Gamma_s$ is the maximal subset of $\widetilde{A}_2$ satisfying the nested condition \eqref{nested-condition} relative to $s$. 
It is clear that $sut, tu, u$ and $t$ are in $\widetilde{A}_2 \setminus \Gamma_s$. Therefore, the subgroup generated by elements in $(\widetilde{A}_2 \setminus \Gamma_s) \cup {\rm Stab}_{ H_s^+}$ is $\widetilde{A}_2$.

Although $\widetilde{A}_2$ does not satisfy the nested condition, it admits a torsion-free subgroup of finite index that does. 
 It is known that $\widetilde{A}_2$ can be realized as the semidirect product $\widetilde{A}_2 = \mathbb{Z}^2\rtimes A_2$ (see, for example, \cite{jushenko2018decompositions}). The generators of $\mathbb{Z}^2$ may be chosen as $\alpha = usts$ and $\beta = sust$. 
 The subgroup $\mathbb{Z}^2$ is  torsion-free and of finite index. Moreover, it is easy to check that 
    $s\mathbb{Z}^2s = \mathbb{Z}^2$ and $\mathbb{Z}^2\cap {\rm Stab}_{H_s^+}=\langle \alpha \beta \rangle$. In this case, the multiplier $m_{\mathbb{Z}^2}^s$ on $\mathbb{Z}^2$ defined in \eqref{eq: multiplier on W0} (with $W_0=\mathbb{Z}^2$) coincides with the classical $v$-directional Hilbert transform $\mathbb{H}_v$ on $\mathbb{Z}^2$, where $v=(1,1)$. Extending this Hilbert transform to $\widetilde{A}_2$  via \eqref{eq: def of general symbol} and applying 
    \ref{theoremB}, we obtain that $\mathbb{H}_v\rtimes \rm{id}_{A_2}$ is bounded on $L_p(\mathcal{L}\widetilde{A}_2)$ for every $1<p<\infty$. 
\end{ex}

As mentioned earlier, the group $\widetilde{A}_2$ has property (F$\mathbb{R}$), as does any Coxeter group such that ${\rm m}_{st}\neq \infty$ for all $s,t\in S$. 
In particular, such groups cannot act on real trees without global fixed points. Then the multiplier defined in \eqref{eq: def of general symbol} falls outside the scope of \cite{gonzalez2022non-commutative}. In the case of $\widetilde{A}_2$, 
 the boundedness of $\mathbb{H}_v\rtimes \rm{id}_{A_2}$ can also be deduced from \cite[Theorem A]{ParKei16}, where the authors studied twisted Hilbert transforms on semidirect products of $\mathbb{R}^n$ (or $\mathbb{Z}^n$) with a discrete group.
However, one can easily construct non-affine Coxeter groups admitting semidirect product decompositions (see \cite{Gal05, BonDyer10} for details). In such cases,  the multiplier defined in \eqref{eq: def of general symbol} on these groups arises as  a semidirect product extension of a multiplier on a parabolic subgroup,
as in the case of $\widetilde{A}_2$. Unlike the affine case, however, the relevant parabolic subgroup is no longer isomorphic to $\mathbb{Z}^n$. Therefore, the boundedness of the Hilbert transforms on these Coxeter groups defined by  \eqref{eq: def of general symbol}  cannot be treated using \cite[Theorem A]{ParKei16}. 
We leave these cases to the interested reader.

\section{Hilbert transforms on groups acting on buildings}\label{section;building}
Buildings were introduced by Tits in the 1950s with the aim of understanding complex semisimple Lie groups from a geometric viewpoint. There are several equivalent definitions of buildings. One classical approach is to view them as simplicial complexes equipped with a family of subcomplexes satisfying certain axioms. In the modern building theory, there is also an important combinatorial approach which is purely axiomatic, where buildings are viewed as $W$-metric spaces with $W$ a Coxeter group. In this paper, we use the latter definition. 

On the one hand, some buildings
 can be viewed as higher-dimensional trees. Motivated by results on groups acting on $\mathbb{R}$-trees in \cite{gonzalez2022non-commutative}, we
consider Hilbert transforms on groups acting on buildings. A key motivation for this setting comes from the fact that ${\rm GL}(n,\mathbb{Z})$ (or ${\rm SL}(n,\mathbb{Z})$), $n\geq 3$, are canonical examples of groups with Kazhdan’s property (T) and therefore have property (F$\mathbb{R}$) (i.e. every isometric action on a $\mathbb{R}$-tree is trivial). Nevertheless, these groups admit non-trivial actions on affine buildings of type $\widetilde{A}_n$ (see \cite[Section 6.9]{abramenko2008buildings} or \cite[Section 9.2]{ronan2009lectures}). The main tool \eqref{eq: Cotlar} used to establish $L_p$-boundedness ($1<p<\infty$) requires that the underlying building on which the group acts satisfies a certain nested condition, in analogy with the Coxeter group case discussed in the previous section. However, affine buildings of type $\widetilde{A}_n$ do not satisfy this condition. Along this line of research, it is worth mentioning a recent work of Parcet, de la Salle and Tablate  \cite{PardelaSalTablate24}, where it is shown that there are no $L_p$-bounded  idempotent Fourier multipliers  on ${\rm SL}(n,\mathbb{R})$ for any $n\geq 2$. 
However, it is an independent question whether there exist Hilbert transforms on ${\rm SL}(n,\mathbb{Z})$, viewed as a lattice of ${\rm SL}(n,\mathbb{R})$. The answer is positive for the case $n=2$ (see \cite{gonzalez2022non-commutative}), but it remains open whether there exist $L_p$-bounded Hilbert transforms (subject to some algebraic relations) on ${\rm SL}(n,\mathbb{Z})$ for $n\ge 3$.

On the other hand, buildings are closely related to Coxeter groups, and we will see later that every Coxeter complex can be viewed as a (thin) building. Therefore,  the results in this section generalize those obtained in the previous subsection for Coxeter groups satisfying \eqref{nested-condition} relative to a certain generator.

\subsection{Preliminaries on buildings}\label{subsec: buildings}
We now recall the definition and some basic properties of buildings; the definitions mentioned in this subsection can be found in \cite[Chapter 5]{abramenko2008buildings}. 

\begin{defn}\label{def: Buildings}
	A \textit{building of type $(W,S)$}, where $(W,S)$ is a Coxeter system, is a pair $(\Delta, \delta)$ consisting of a non-empty set $\Delta$, whose elements are called \textit{chambers}, and a map $\delta: \Delta\times \Delta\to W$, which is called the \textit{Weyl distance function} satisfying the following axioms: for $C, D\in \Delta$,
\begin{enumerate}
    \item[(B1)] $\delta(C, D) = 1$ if and only if $C = D$.
    \item[(B2)] If $\delta(C,D) = w$ and $E\in\Delta$ with $\delta(E,C) = s\in S$, then $\delta(E,D) = sw$ or $w$. In particular,  if $l(sw)> l(w)$, then $\delta(E,D) = sw$.
    \item[(B3)]If $\delta(C,D) = w$, then for every $s\in S$ there exists a chamber $E\in \Delta$ such that $\delta(E,C) = s$ and $\delta(E,D) = sw$.
\end{enumerate}	
	
\end{defn}

The above axioms are analogues of those for metric spaces, and for this reason buildings are sometimes called $W$-metric spaces. By axioms (B1) and (B3), it follows that for $s\in S$,
$$\delta(C,D) =s \Longleftrightarrow \delta(D,C) = s, \;\; \text{ for } C, D\in \Delta.$$

\begin{defn}
	Let $(\Delta, \delta)$ be a building of type $(W,S)$, and let $C,D\in \Delta$. If $\delta(C,D) = s\in S$, then $C$ and $D$ are said to be $s$-\textit{adjacent}. If $\delta(C,D)\in \{1,s\}$ for some $s\in S$, then $C$ and $D$ are said to be $s$-\textit{equivalent}. 
\end{defn}

In fact, every Coxeter complex can be viewed as a building.
\begin{ex}\label{ex: Coxeter as a building}
	Let $(W,S)$ be a Coxeter system. Define a map $\delta_W: W\times W \to W$ by $\delta_W(g,h) := g^{-1}h$. It is straightforward to verify that axioms (B1)-(B3) are satisfied. Moreover, for each $s \in S$, every $s$-equivalence class contains exactly two elements of $W$. Hence $(W,\delta_W)$ is a building of type $(W,S)$, which we call a \emph{standard thin building}. 
    Furthermore, for any three chambers $C$, $D$ and $E$, one has $\delta_W(C,D)=\delta_W(C,E)\delta_W(E,D)$. Note that this property does not hold in general buildings.
\end{ex}

Note that the function defined by
$$d(C,D) := l(\delta(C,D))$$
where $C,D\in \Delta$, gives a metric on $(\Delta, \delta)$ between chambers (see, for instance, \cite[Corollary 5.17]{abramenko2008buildings}).  We record a useful property of the Weyl distance function, which will be used frequently in the subsequent computations.
\begin{prop}\cite[Corollary 5.17]{abramenko2008buildings}\label{prop;delta-equal-inverse}
	Let $C$ and $D$ be two chambers of $(\Delta, \delta)$. Then $\delta(C,D) = \delta(D,C)^{-1}$.
\end{prop}

We next introduce another important class of objects in buildings, namely apartments. 

\begin{defn}
	Let $(\Delta, \delta)$ and $(\Delta', \delta')$ be buildings of type $(W,S)$. 
 An \textit{isometry} from $\Delta'$ to $\Delta$ is a map
    $\phi: (\Delta', \delta')\to (\Delta, \delta)$ such that 
	$$\delta(\phi(C),\phi(D)) = \delta'(C,D),$$
	for all $C, D\in \Delta'$. 
    A subset $\Sigma \subseteq \Delta$ is called an \textit{apartment} if there exists an isometry $\phi\colon (W,\delta_W)\to(\Delta, \delta)$ such that $\Sigma=\phi(W)$.
\end{defn} 

\begin{rem}\label{rem: existence of apartment}
 By \cite[Theorem 5.73]{abramenko2008buildings}, every isometry from a subset of $W$ to $\Delta$ extends to an  isometry from $W$ to $\Delta$. It follows that every subset of $\Delta$ that is isometric to a subset of $W$ is contained in an apartment. In particular, for any two chambers $C$, $D\in \Delta$, the subset $\{C,D\}$ is isometric to $\{1, w\}$, where $w=\delta(C,D)$. Hence, there exists an apartment of $\Delta$ containing both $C$ and $D$.   
\end{rem}

\begin{rem}\label{rem: unique s-adjacency}
Let $\Sigma $ be an apartment in $\Delta$, i.e. $\Sigma=\phi(W)$ for an isometry $\phi$ from $W$ to $\Delta$.  By the definition of isometries, it is clear that $\phi$ preserves adjacency. Then
Example \ref{ex: Coxeter as a building} implies that for every chamber $C\in \Sigma$ and every generator $s\in S$, there is a unique chamber in $\Sigma$ that is  $s$-adjacent to $C$. We denote it by $s_\Sigma C$.
\end{rem}

The following proposition is implicit in \cite[Section 5.5.2]{abramenko2008buildings}. Since the definition of apartments used there is different  from (but equivalent to) ours, we provide a self-contained proof. 
\begin{prop}\label{prop; surj-isometry-apartment-Complex}
	Let $(\Delta, \delta)$ be a building of type $(W,S)$, and let $\Sigma $ be an apartment in $\Delta$. Write $\Sigma=\phi(W)$ for an isometry $\phi: W \rightarrow \Delta$.
     Fix a chamber $C'\in \Sigma$ and define a map $\psi_{\Sigma,C'}: \Delta\to W$ by $\psi_{\Sigma,C'}(C) := \delta(C', C)$. Then $(\psi_{\Sigma,C'})|_{\Sigma}$ is  a surjective isometry.  More precisely,
     $$\psi_{\Sigma,C'}(\Sigma) = W\quad\text{and}\quad\delta_W(\psi_{\Sigma,C'}(C), \psi_{\Sigma,C'}(D)) = \delta(C,D).$$
	Moreover, for any $C,D,E\in \Sigma$, one has $\delta(C,E)\delta(E,D) = \delta(C,D)$. In particular, if $C' = \phi(1)$, then $\psi_{\Sigma,C'} = \phi^{-1}$. 
\end{prop}

\begin{proof}
Let $C'=\phi(w')$ for some $w'\in W$. To show surjectivity of $\psi_{\Sigma,C'}$, let
$w\in W$. Then 
$$\psi_{\Sigma,C'} (\phi (w'w))=\delta (\phi(w'), \phi (w'w))=w.$$
Hence $\psi_{\Sigma,C'}$ is surjective. 
 Let $C, D\in \Sigma$ and write $C=\phi(w_1)$ and $D = \phi(w_2)$ for $w_1, w_2\in W$. Then by the definition, we have 
	$$\psi_{\Sigma,C'}(C) = \delta(C', C) = \delta(\phi(w'),\phi(w_1)) = w'^{-1}w_1,$$
	and similarly $\psi_{\Sigma,C'}(D) = w'^{-1}w_2$. Therefore,
	$$\delta_W(\psi_{\Sigma,C'}(C),\psi_{\Sigma,C'}(D)) =(w'^{-1}w_1)^{-1}w'^{-1}w_2 = w_1^{-1}w_2 = \delta_W(w_1,w_2) = \delta(C,D),$$
    so $\psi_{\Sigma,C'}$ is an isometry.
	By Proposition \ref{prop;delta-equal-inverse}, we have 
	$$\psi_{\Sigma,C'}(C)^{-1}\psi_{\Sigma,C'}(D) = \delta(C',C)^{-1}\delta(C',D) = \delta(C,C')\delta(C',D).$$
	Hence $\delta(C,C')\delta(C',D) = \delta(C,D)$. By replacing $C'$ with $E$, we get the second statement. Finally, if $C'=\phi(1)$ and $C = \phi(w)$, then $\psi_{\Sigma,C'}(C) =\delta(\phi(1),\phi(w)) = w$. Hence,
	$$\phi(\psi_{\Sigma,C'}(C)) = C\; \text{and}\; \psi_{\Sigma,C'}(\phi(w)) = w,$$
    which implies $\psi _{\Sigma,C'}= \phi^{-1}$.

\end{proof}

\begin{rem}\label{rem: change phi}
The map $\psi_{\Sigma,C'}$ in the above proposition is defined on the building $\Delta$ and takes values in $W$. It depends on the apartment $\Sigma$ and a chosen chamber $C'$. This type of maps is known as \textit{retractions}. 
For any chamber $C\in\Sigma=\phi(W)$, the map $(\psi_{\Sigma,C})|_{\Sigma}\colon \Sigma\to W$ is a surjective isometry satisfying $\psi_{\Sigma,C}(C)=1$. Hence its inverse $(\psi_{\Sigma,C}|_{\Sigma})^{-1}\colon W\to\Sigma$ is a surjective isometry sending $1$ to $C$. Thus, after fixing an apartment $\Sigma$, one may choose an identification of $W$ with $\Sigma$ so that the identity element of $W$ corresponds to any prescribed chamber of $\Sigma$.
\end{rem}

By definition, every apartment in a building of type $(W,S)$ is isometric to $W$. 
It follows that any two apartments are isometric. We refer the reader to  \cite[Corollary 5.68]{abramenko2008buildings} for a proof of the following corollary. 

\begin{cor}
 Let $\Sigma_1$ and $\Sigma_2$ be apartments in $\Delta$. 
Then there exists a surjective isometry $\phi:\Sigma_1\to\Sigma_2$ such that 
	$\phi(C) = C$,
	for every chamber $C\in \Sigma_1\cap \Sigma_2$. More precisely, if $C$ is a chamber in $\Sigma_1\cap\Sigma_2$, then one may choose surjective isometries
		\[
		\phi_1:W\to\Sigma_1
		\quad\text{and}\quad
		\phi_2:W\to\Sigma_2,
		\]
		with $\phi_1(1)=\phi_2(1)=C$, and define
		$\phi=\phi_2\circ\phi_1^{-1}$.
		Moreover, $\phi$ preserves adjacency types; that is, it sends $s$-adjacent chambers to $s$-adjacent chambers for every $s\in S$.   
\end{cor}

    
Since the elements of $W$ are in bijection with  the chambers of the Coxeter complex $\Sigma(W,S)$ via $w \mapsto wK$, where $K$ is the fundamental chamber of $\Sigma(W,S)$, we may equip $\Sigma(W,S)$ with the $W$-metric  
$$\delta_{\Sigma(W,S)} (w_1K, w_2K)=\delta_{W} (w_1, w_2)=w_1^{-1}w_2.$$ It follows that every apartment in a building of type $(W,S)$ is isometric to the Coxeter complex $\Sigma(W,S)$.  Then we may denote an apartment by $\phi(\Sigma(W,S))$ for some isometry $\phi$.
Now we introduce the notions of walls and roots in buildings.

\begin{defn}
	Let $(\Delta, \delta)$ be a building of type $(W,S)$. A subset of $\Delta$ is called a \textit{root} (resp. \textit{wall}) if it is isometric to a root (resp. wall) in $\Sigma(W,S)$. 
    For any chamber $C\in \Delta$, if $C$ is identified with $wK$, then its \textit{$s$-face} is defined to be the image of $K_s$ in the chamber $wK$ under this identification. 
\end{defn}
	As in the case of Coxeter complexes, roots in buildings also admit a description in terms of the $W$-metric on chambers.

\begin{lem}\label{lem;root in building}
	Let $(\Delta, \delta)$ be a building of type $(W,S)$, and let $H^\dagger$ be a root in $\Delta$, where $\dagger\in \{+, -\}$.
	Then there is a reflection $r = wsw^{-1}\in W$ with $w\in W$, $s\in S$ and an isometry 	\[
		\phi:\Sigma(W,S)\to\Sigma
		\]
		onto an apartment $\Sigma$ such that $H^\dagger=\phi(H_r^+)$ or $H^\dagger=\phi(H_r^-)$.
        Let $C = \phi(wK)$ (with the same $w$ as above), and let $s_{\Sigma}C\in \Sigma$ denote the chamber $s$-adjacent to $C$.
    \begin{enumerate}
        \item [(i)]  If $H^\dagger= \phi(H_r^{+})$, then one obtains 
	\begin{equation*}\label{equation;root in building}
		\phi(H_r^{+})= H(C, s_{\Sigma}C) := \{D\in \Sigma: d(D,C)< d(D,s_{\Sigma}C)\}\quad \text{if } l(w)<l(ws),
	\end{equation*}
	and
	\begin{equation*}
		\phi(H_r^{+})= H(s_{\Sigma}C, C) := \{D\in \Sigma: d(D,C)> d(D,s_{\Sigma}C)\}\quad \text{if } l(ws)<l(w).
	\end{equation*}

    \item[(ii)] If $ H^\dagger= \phi(H_r^{-})$, 
    then $H^\dagger = \{D\in \Sigma: D\notin {\phi(H_r^{+})}\}$.
    \end{enumerate}

\end{lem}
\begin{proof}
	(i) 
   Recall that in the Coxeter complex, if  $l(w)<l(ws)$, then by Remark \ref{rem: expression of half spaces} the positive root associated with the wall $H_r$ is given by
   $$H_r^{+} = \{gK\in \Sigma(W,S): d(gK,wK)<d(gK,wsK)\}.$$
   Applying the isometry $\phi$, we obtain
  $$ \phi (H_r^{+}) =  \{D\in \Sigma: d(D,C)<d(D,\phi(wsK))\}.$$ 
  It remains to show that $\phi(wsK)=s_\Sigma C$.
	Since $\phi$ is an isometry, we have 
	$$\delta(C,\phi(wsK))= \delta_W(wK,wsK)=s=\delta(C,s_{\Sigma}C).$$
	Thus both $\phi(wsK)$ and $s_{\Sigma}C$ are chambers in $\Sigma$ that are $s$-adjacent to $C$. By uniqueness of the $s$-adjacent chamber in an apartment, we conclude that $\phi(wsK) = s_{\Sigma}C$. Similarly, one can show $\phi (H_r^{+})=\{D\in \Sigma: d(D,C)> d(D,s_{\Sigma}C)\}$ for the case $l(w)>l(ws)$.

    (ii) This follows directly from (i) and the fact that $\Sigma(W,S)=H_r^+ \cup H_r^{-}$.
\end{proof}


There is a rich theory of groups acting on buildings. Two typical types of actions considered in the literature are strongly transitive actions and Weyl transitive actions \cite[Chapter 6]{abramenko2008buildings}.
	In this paper, however, we only assume that the group action is by isometries.

\begin{defn}
	Let $(\Delta, \delta)$ be a building of type $(W,S)$ and let $G$ be a locally compact group. We say that $G$ acts by \textit{isometries} on $(\Delta, \delta)$ if 
	$$\delta(gC,gD) = \delta(C,D)$$
    for all $g\in G$ and all $C,D\in \Delta$.
\end{defn}

We refer the reader to \cite[Chapter 6]{abramenko2008buildings} for further details on groups acting on buildings.

\subsection{Buildings satisfying the nested condition}
 The nested condition played a key role in establishing the Cotlar identity for Coxeter groups in the previous section. We now extend this property to buildings.
Let $(\Delta, \delta)$ be a building of type $(W,S)$, where $W$ satisfies \eqref{nested-condition} relative to a certain generator. In this case, we say that $(\Delta, \delta)$ is a \textit{building satisfying the nested condition} relative to the same generator. 

\medskip 

 We prove in the following that inclusion  of roots in buildings with the nested condition satisfies a certain ``transitivity''. This is the key proposition in establishing that the multiplier defined later in \eqref{eq: Def of symbol for actions on buildings} satisfies the Cotlar identity. 

\begin{prop}\label{pro;building-transitivity}
    Let $(\Delta, \delta)$ be a building of type $(W,S)$, where $W$ satisfies the nested condition relative to $u\in S$.
    Let $C,D,E\in\Delta$, and let $\Sigma = \phi(\Sigma(W,S))$, $\Sigma_1 = \phi_1(\Sigma(W,S))$ be apartments containing $C,D$ and $D,E$ respectively with 
    \[
		\phi(wK)=C,\quad \phi(K)=D=\phi_1(K),\quad \phi_1(w_1K)=E
		\]
        for some $w$, $w_1\in W$. Let $\Sigma_2 = \phi_2(\Sigma(W,S))$ be an apartment containing $C,E$ with $\phi_2(K) = C$ and $\phi_2(w_2K) = E$ for some $w_2\in W$. Then the following statements hold:
    \begin{enumerate}
        \item [(i)] If there exist signs $\dagger,\dagger'\in\{+,-\}$ 
        such that 
        $\phi(w\cdot H_{u}^\dagger)\subsetneq \phi(H_{u}^+)$ and $\phi_1(H_{u}^+)\subsetneq \phi_1(w_1\cdot H_{u}^{\dagger'})$, then we have 
        $$\phi_2(H_{u}^\dagger)\subsetneq \phi_2(w_2\cdot H_{u}^{\dagger'}).$$
        \item[(ii)]  If there exist signs $\dagger,\dagger'\in\{+,-\}$ 
        such that  $\phi(w\cdot H_{u}^\dagger)\subsetneq \phi(H_{u}^-)$ and $\phi_1(H_{u}^-)\subsetneq \phi_1(w_1\cdot H_{u}^{\dagger '})$, then we have
        $$\phi_2(H_{u}^\dagger)\subsetneq \phi_2(w_2\cdot H_{u}^{\dagger '}).$$
 \item [(iii)] If $\phi(w\cdot H_{u}^+)= \phi(H_{u}^-)$ and $\phi_1(H_{u}^-)\subsetneq \phi_1(w_1\cdot H_{u}^\pm)$, then  we have 
        $$\phi_2(H_{u}^+)\subsetneq\phi_2(w_2\cdot H_{u}^\pm).$$

    \item [(iv)] If $\phi(w\cdot H_{u}^\pm)\subsetneq \phi(H_{u}^+)$ and $\phi_1( H_{u}^+)= \phi_1(w_1\cdot H_{u}^-)$, then we have
    $$\phi_2(H_{u}^\pm)\subsetneq\phi_2(w_2\cdot H_{u}^-).$$
    \end{enumerate}
    
\end{prop}
\begin{proof}
    (i) Let $T_u=\{s\in S: \mathrm{m}_{su} = 2\}$ and $N_u \subseteq W$ be the set of elements whose reduced expressions neither start nor end with $u$.
    Suppose that $\phi(w\cdot H_{u}^+)\subsetneq \phi(H_{u}^+)$ and $\phi_1(H_{u}^+)\subsetneq \phi_1(w_1\cdot H_{u}^+)$. Since $\phi$ and $\phi_1$ are surjective isometries, the two inclusions are equivalent to 
    $$w\cdot H_{u}^+\subsetneq H_{u}^+ \subsetneq w_1\cdot H_{u}^+.$$
   By Proposition~\ref{prop; equi-char-Nest-algebraic-forms}, it follows that $w = hu$ and $w_1 = uh_1$ for some $h,h_1\in N_u\setminus W_{T_u}$. Let $h = s_1 \ldots s_k$ be a reduced expression with $k\geq 1$. Then $$\delta(C,D)=\delta(\phi(wK),\phi(K))=w^{-1} = us_k\ldots s_1$$ and $$\delta(D,E)=\delta(\phi_1(K), \phi_1(w_1K))=w_1 = uh_1.$$
   We claim that $\delta(C,E)=w_2\in u( N_u\setminus W_{T_u})$. 
   If that is the case, applying Proposition \ref{prop; equi-char-Nest-algebraic-forms} again, we have 
    $$w_2\cdot H_u^-\subsetneq H_u^- \Longleftrightarrow  H_u^+\subsetneq w_2\cdot H_u^+.$$
    Since $\phi_2$ is an isometry, it implies that  $\phi_2(H_u^+)\subsetneq \phi_2(w_2\cdot H_u^+)$. It remains to prove the claim.
For the above given reduced expression of $h$, we first shuffle the letters according to the following rule: whenever some $s_j$ with $1\leq j\leq k$ commutes with $u$, we use the $M$-operations to move $s_j$ to the position as closest as possible to the end of the word. 
This yields another reduced expression of $h$:
$$h=s'_k\ldots s'_1.$$  
Since $h\notin W_{T_u}$, there exists a  minimal index $1\le i_0\le k$ such that $\mathrm{m}_{s'_{i_0}u}\neq 2$ and $\mathrm{m}_{s'_{i}u}= 2$ for all $ i<i_0$. 
By the assumption and Proposition \ref{prop;W-satisfies-NestC-equi-to-WRA}, for any $u\ne s\in S$, $\mathrm{m}_{su}\in\{2,\infty\}$, and then $\mathrm{m}_{s'_{i_0}u}= \infty$. In other words, $s'_{i_0}$ and $u$ are free. 

If $i_0=1$, then $s_1'$ and $u$ are free. In this case, we have $l(s'_1 w_1)>l(w_1)$ since $w_1$ admits a reduced expression starting with $u$. Moreover, by the construction of the reduced expression
		$h=s'_k\cdots s'_1$,
for $s_2'$, either $s_2'$ and $u$ are free, or $s_2'$ does not commute with $s_1'$. In either case, we have $l(s_2's'_1 w_1)>l(s_1'w_1)$. It follows by induction that $l(s'_j\ldots s'_1 w_1)>l(s'_{j-1}\ldots s_1'w_1)$ for any $2\leq j\leq k$. Furthermore, $l(us'_k\ldots s'_1 w_1)>l(s'_{k}\ldots s_1'w_1)$.
Set $C_j=\phi (s_j'\ldots s_1'K)\in \Sigma$ for $1\le j\le k$. Since $\delta (C_1, D)=s_1'$ and $\delta (D,E)=w_1$, axiom {\rm(B2)} in Definition~\ref{def: Buildings} implies that 
$$\delta(C_1, E)=s_1'w_1=\delta (C_1, D)\delta (D,E).$$ 
Applying {\rm(B2)} inductively, we obtain $\delta (C_j, E)=s'_j\ldots s'_1 w_1=\delta(C_j, C_{j-1})\delta (C_{j-1}, E)$ for any $2\leq j\leq k$, and  
$$\delta (C, E)=us'_k\ldots s'_1w_1=us'_k\ldots s'_1 uh_1.$$
Since $h_1\in N_u\setminus W_{T_u}$ and $s_1'$ is free with $u$, we have $w_2 = \delta (C, E)\in u( N_u\setminus W_{T_u})$.

If $i_0\geq 2$, then $s_1'$ commutes with $u$. Consequently, depending on the reduced expression of $h_1$, either $l(s'_1 w_1)>l(w_1)$ or $l(s'_1 w_1)<l(w_1)$. Hence $\delta(C_1, E)=s'_1w_1$ or $w_1$. Since $s_1'$ commutes with $u$ and $w_1$ admits a reduced expression starting with $u$, both $s'_1w_1$ and $w_1$ admit a reduced expression starting with $u$. Repeating this argument for $s'_2,\ldots,s'_{i_0-1}$, we obtain a sequence of natural numbers  $1\le \ell_1\le \ldots \le \ell_j<i_0$ such that 
$$\delta(C_{i_0-1}, E)= s'_{\ell_j}\ldots s'_{\ell_1}w_1 =  s'_{\ell_j}\ldots s'_{\ell_1}uh_1,$$
and $\delta(C_{i_0-1}, E)$ admits a reduced expression starting with $u$. Since $h\in N_u$, any $s'_{\ell_i}$ with $1\leq i<i_0$ above is not equal to $u$. Moreover, it follows from $h_1\in N_u\setminus W_{T_u}$ that $s'_{\ell_j}\ldots s'_{\ell_1}w_1$ does not admit reduced expressions ending with $u$. Then by a similar argument as in the case of $i_0=1$, we conclude that $\delta (C, E)$ is in $u( N_u\setminus W_{T_u})$. 

    Now consider the case $\phi(w\cdot H_{u}^-)\subsetneq \phi(H_{u}^+)$ and $\phi_1(H_{u}^+)\subsetneq \phi_1(w_1\cdot H_{u}^-)$, which is
    equivalent to 
    $$w\cdot H_{u}^-\subsetneq H_{u}^+ \subsetneq w_1\cdot H_{u}^-.$$
By Proposition \ref{prop; equi-char-Nest-algebraic-forms}, there exist $h,h_1\in N_u\setminus W_{T_u}$ such that $w = h$ and $w_1 = uh_1u$. It remains to determine whether reduced expressions of $w_2$ start or end with $u$. By a similar argument as in the previous case,  we obtain $w_2\in ( N_u\setminus W_{T_u})u$. 
Then applying Proposition \ref{prop; equi-char-Nest-algebraic-forms} again, we have 
    $$w_2\cdot H_u^+\subsetneq H_u^+ \Longleftrightarrow H_u^-\subsetneq w_2\cdot H_u^-,$$
    and then $\phi_2(H_u^-)\subsetneq \phi_2(w_2\cdot H_u^-)$.

    For the case $\phi(w\cdot H_{u}^+)\subsetneq \phi(H_{u}^+)$ and $\phi_1(H_{u}^+)\subsetneq \phi_1(w_1\cdot H_{u}^-)$, we have 
    $$w\cdot H_{u}^+\subsetneq H_{u}^+ \subsetneq w_1\cdot H_{u}^-.$$ 
    By Proposition \ref{prop; equi-char-Nest-algebraic-forms}, there exist $h,h_1\in N_u\setminus W_{T_u}$ such that $w = hu$ and $w_1 = uh_1u$. Similarly, we obtain $w_2\in u( N_u\setminus W_{T_u})u$. 
    Then it implies that 
    $$w_2\cdot H_u^+\subsetneq H_u^- \Longleftrightarrow H_u^+\subsetneq w_2\cdot H_u^-,$$
    and then $\phi_2(H_u^+)\subsetneq \phi_2(w_2\cdot H_u^-)$.

For the last case $\phi(w\cdot H_{u}^-)\subsetneq \phi(H_{u}^+)$ and $\phi_1(H_{u}^+)\subsetneq \phi_1(w_1\cdot H_{u}^+)$, we have 
$$w\cdot H_{u}^-\subsetneq H_{u}^+ \subsetneq w_1\cdot H_{u}^+.$$ 
By Proposition \ref{prop; equi-char-Nest-algebraic-forms}, there exist $h,h_1\in N_u\setminus W_{T_u}$ such that $w = h$ and $w_1 = uh_1$. Similarly, we obtain $w_2\in N_u\setminus W_{T_u}$. Then it implies that 
$$w_2\cdot H_u^-\subsetneq H_u^+ \Longleftrightarrow H_u^-\subsetneq w_2\cdot H_u^+,$$
    and then $\phi_2(H_u^-)\subsetneq \phi_2(w_2\cdot H_u^+)$.

    (ii) 
    The argument for part (ii) is analogous to that of part (i), and we omit the details.

    (iii) In this case, we have $w=uh$ with $h\in W_{T_u}$, while $w_1$ is either of the form $h_1u$ or $h_2$, where $h_1,h_2\in N_u\setminus W_{T_u}$. Then a similar argument yields that  $\delta(C,E)$ is  either in $u( N_u\setminus W_{T_u})u$ or in $u( N_u\setminus W_{T_u})$.  
    The conclusion then follows by applying Proposition~\ref{prop; equi-char-Nest-algebraic-forms}.

    (iv) In this case, $w$ is either $h_1u$ or $h_2$, where $h_1,h_2\in N_u\setminus W_{T_u}$, and $w_1 = uh$ with $h\in W_{T_u}$. Then part (iii) implies that 
    $$H_u^+\subsetneq w_2^{-1}\cdot H_u^{\pm}\Longleftrightarrow w_2\cdot H_u^+\subsetneq H_u^{\pm} \Longleftrightarrow H_u^{\pm}\subsetneq w_2\cdot H_u^-,$$
    and the assertion follows.
\end{proof}

\begin{rem}
    When restricting to the standard thin building of type $(W,S)$ which satisfies the nested condition, the above proposition becomes easier to prove. The main difficulty in Proposition \ref{pro;building-transitivity} arises from the fact that,  in a general building, if $\delta(C,D) = w$, $\delta(D,E) = w_1$ and $\delta(C,E) = w_2$, one does not necessarily have $w_2=ww_1$.
\end{rem}

\subsection{Hilbert transforms on groups acting on buildings satisfying the nested condition}
In this subsection, we first define the Hilbert transform on groups admitting actions on buildings  as a Fourier multiplier. We then state our main result, showing that when the building satisfies the nested condition, the symbol of the defined Fourier multiplier satisfies \eqref{eq: Cotlar}. Throughout this subsection, let $(\Delta,\delta)$ be a building of type $(W,S)$, where $W$ satisfies \eqref{nested-condition} relative to $u\in S$.

Fix a chamber $C_0\in \Delta$. Define a map $\widetilde m$ on $\Delta$ by: for any chamber $D\in \Delta$,

\begin{equation}\label{eq: symbol on buildings}
 \widetilde m(D) := \left\{\begin{array}{l}
		1, \quad \text{if} \; D\in {\phi(H_u^+)},\\
		-1, \quad \text{if} \;  D\in {\phi(H_u^-)}, 
	\end{array} \right. 
\end{equation}
where $\Sigma =\phi(\Sigma(W,S))$ is any apartment containing $D$ and $C_0$ (the existence of such an apartment follows from Remark \ref{rem: existence of apartment}). The isometry $\phi$ is chosen so that $\phi(K) = C_0$, which is possible by Remark \ref{rem: change phi}. 

Note that there may be more than one apartment containing $D$ and $C_0$, so it remains to show that $\widetilde m$ is well-defined. 
Recall that for the chosen $\phi$, Lemma \ref{lem;root in building} yields $\phi(H_u^+)=\{D\in \Sigma: d(D,C_0)< d(D,u_{\Sigma}C_0)\}$ and $\phi(H_u^-)=\{D\in \Sigma: d(D,C_0)>d(D,u_{\Sigma}C_0)\}$. Since $D,C_0, u_{\Sigma}C_0$ are in the same apartment, Proposition \ref{prop; surj-isometry-apartment-Complex} implies $\delta (D, u_{\Sigma}C_0)=\delta (D, C_0)\delta (C_0, u_{\Sigma}C_0)=\delta (D, C_0)u$. For any $D\in \Delta$, there exists $w\in W$ such that $\delta(C_0,D) = w$. 
We have $$D\in \phi(H_u^+) \Longleftrightarrow l(w)<l(uw),$$
since $d(D,C_0)=l(\delta(D,C_0))=l(w^{-1})<d(D,u_{\Sigma}C_0)=l(\delta(D,u_{\Sigma}C_0))=l(w^{-1}u)$. Similarly, we have 
$$D\in \phi(H_u^-) \Longleftrightarrow l(w)>l(uw).$$
Let $\Sigma_1 = \phi_1(\Sigma(W,S))$ and $\Sigma_2=\phi_2(\Sigma(W,S))$ be two apartments containing both $D$ and $C_0$, with $\phi_1(K) = C_0 = \phi_2(K)$. Then we have $$D\in \phi_1(H_u^+)\iff l(w)<l(uw)\iff D\in\phi_2(H_u^+),$$
and $$D\in \phi_1(H_u^-)\iff l(w)>l(uw)\iff D\in\phi_2(H_u^-).$$
Thus $\widetilde{m}$ is independent of the choice of apartments.

By the above discussion, the function defined in \eqref{eq: symbol on buildings} is closely related to the multiplier $m_W^u$ defined on $W$ in \eqref{eq: def of symbol on W} and to the retraction defined in Proposition \ref{prop; surj-isometry-apartment-Complex}. For any $D\in \Delta$ with $\delta(C_0,D) = w\in W$, it follows from \eqref{eq;algebric-def-symbol-w RA-Cg} that
$$
\widetilde{m} (D)= m_W^u (\psi_{C_0}(D)) = m_W^u(w).
$$

Now let $G$ be a locally compact group that acts on $(\Delta, \delta)$ isometrically. The function defined on the building in \eqref{eq: symbol on buildings} induces a function  $m_G^u$ on $G$ by
\begin{equation}\label{eq: Def of symbol for actions on buildings}
    m_G^u(g) := \widetilde m(gC_0) = \left\{\begin{array}{l}
		1, \quad \text{if} \; gC_0\in \phi(H_u^+),\\
		-1, \quad \text{if} \;  gC_0\in \phi(H_u^-), 
	\end{array} \right.
\end{equation}
where $\Sigma = \phi(\Sigma(W,S))$ is any apartment containing $C_0$ and $gC_0$ such that $\phi(K) = C_0$.
For any $g\in G$ with $\delta(C_0,gC_0) = w\in W$, the function $m_G^u$ admits an algebraic description: 
\begin{equation}\label{eq;Def of algebric form of m on buildings}
    m_G^u(g) = \left\{\begin{array}{l}
		1, \quad \text{if } \;  l(w)<l(uw),\\
		-1, \quad \text{if } \;   l(w)>l(uw). 
	\end{array} \right.
\end{equation}
	We now proceed to the proof of \ref{theoremC}, where we show that the function $m_G^u$ defined above satisfies the Cotlar identity.
\medskip


\begin{proof}[Proof of \ref{theoremC}]
 (i) That $G_0$ is a subgroup of $G$ follows from the fact that the action preserves $\delta$. 
It is clear that $1\in G_0$. Let $g\in G_0$ with $\delta(C_0, gC_0) = w_0\in W_{T_u}$. By Proposition~\ref{prop;delta-equal-inverse} and the isometric action of $G$, we have 
 $\delta(C_0,g^{-1}C_0) = \delta(g^{-1}C_0,C_0)^{-1}=\delta(C_0,gC_0)^{-1} =w_0^{-1}\in W_{T_u}$, so that  $g^{-1}\in G_0$. Now, let $ h\in G_0$ with  $\delta(C_0,hC_0) = w_0'\in W_{T_u}$. Write $w_0=s_1 \ldots s_k$ as a reduced expression with $s_1,\ldots,s_k\in T_u$. 
 It follows  from 
 axiom (B2) of Definition \ref{def: Buildings} 
 that there is some $j\in\{1,\ldots,k\}$ and a sequence of natural numbers $1\le \ell_1\le \ldots\le \ell_j\le k$ such that $\delta(g^{-1}C_0,hC_0) = \delta(C_0,ghC_0) = s_{\ell_1}\ldots s_{\ell_j}w_0'$. Since $w_0', s_i\in W_{T_u}$ for $i= \ell_1,\ldots,\ell_j$, we obtain $gh\in G_0$. Since the orbit map is continuous, we have $\rm{Stab}_{C_0}$ is open in $G$. On the other hand, we have $G_0=\cup_{g_\in G_0} g \,\rm{Stab}_{C_0}$, therefore $G_0$ is open (so closed) in $G$.

 Next we show that the multiplier $m_G^u$ satisfies the Cotlar identity.	
Suppose that $m_G^u(g)\ne m_G^u(h)$ with $m_G^u(g) = 1$ and $m_G^u(h) = -1$. 
Then there are apartments $\Sigma=\phi(\Sigma(W,S))$ and $\Sigma_1=\phi_1(\Sigma(W,S))$ such that $gC_0, C_0\in \Sigma$ and $C_0, hC_0\in \Sigma_1$, with $\phi(K) = C_0 = \phi_1(K)$. Suppose that $\phi(wK) = gC_0$ and $\phi_1(w_1K) = hC_0$ for some $w,w_1\in W$. Since $gC_0\in \phi(H_u^+)$ and $hC_0\in \phi_1(H_u^-)$, we have $wK\in H_u^+$ and $w_1K\in H_u^-$. Since we are assuming $g\notin G_0$, this implies $\delta (C_0, gC_0)= \delta (\phi(K), \phi(wK))=\delta_W(K, wK)=w\notin W_{T_u}$. Applying  Proposition \ref{prop;nested-condition}, for $w$, we obtain that one of the following cases holds:
$$w\cdot H_u^+\subsetneq H_u^+\quad \text{or}\quad w\cdot H_u^-\subsetneq H_u^+.$$
 Similarly, for $w_1$, one of the following cases holds:
 $$w_1\cdot H_u^+ = H_u^-,\quad w_1\cdot H_u^+\subsetneq H_u^- \quad \text{or}\quad w_1\cdot H_u^-\subsetneq H_u^-.$$
For each case the argument is similar and the main tools are Proposition \ref{prop;nested-condition} and Proposition \ref{pro;building-transitivity}. In the following, we give the proof for the  case
$$w\cdot H_u^+\subsetneq H_u^+\quad \text{and}\quad w_1\cdot H_u^+\subsetneq H_u^-,$$
that is
\begin{equation}\label{eq;Main-theorem-building}
    w\cdot H_u^+\subsetneq H_u^+ \subsetneq w_1\cdot H_u^-.
\end{equation}
It follows from $H^{+}_u\subsetneq w^{-1}\cdot H_u^+$ and Proposition \ref{prop;nested-condition} that $w^{-1}K\in H_u^-$. Therefore, $\phi(w^{-1}K)\in \phi (H_u^-)$. Since $\delta(C_0,g^{-1}C_0) = w^{-1}$, $\delta(C_0, \phi(w^{-1}K))= \delta(\phi(K), \phi(w^{-1}K))=w^{-1}$ 
and $C_0, g^{-1}C_0, \phi(w^{-1}K)$ are  in the same apartment $\Sigma$, we deduce that $g^{-1}C_0= \phi(w^{-1}K)\in \phi (H_u^-)$, that is, $m_G^u(g^{-1}) = -1$. Let $w_2 = \delta(gC_0,hC_0)=\delta(C_0,g^{-1}hC_0)$ and $\Sigma_2 = \phi_2(\Sigma(W,S))$ be the apartment containing $gC_0$ and $hC_0$ with $\phi_2(K) = gC_0$. Applying Proposition \ref{pro;building-transitivity}(i) to \eqref{eq;Main-theorem-building}, we obtain 
$$\phi_2(H_u^+)\subsetneq \phi_2(w_2\cdot H_u^-).$$
Thus $H_u^+\subsetneq w_2\cdot H_u^-$, and by Proposition \ref{prop;nested-condition} again we deduce that $w_2 K \in H_u^-$, so that $m_G^u(g^{-1}h) = m_G^u(g^{-1}) = -1$.
Therefore, we get the Cotlar identity.
The case $m_G^u(g) = -1$ and $m_G^u(h) = 1$ can be treated in a similar fashion. 

(ii)  Let $g\in G_0$ with $\delta(C_0, gC_0) = w_0=s_1 \ldots s_k\in W_{T_u}$ as in the first paragraph of the proof of (i), and let $h\in G$ with $\delta(C_0, hC_0)=w_1$ as in the second paragraph of that of (i).
Since $\delta(g^{-1}C_0,hC_0) = \delta(C_0,ghC_0) = s_{\ell_1}\ldots s_{\ell_j}w_1$ with $s_{\ell_i}\in T_u$ for any $1\leq i\leq j$, we have
    $$m_G^u(h) = 1\, (\text{i.e. }l(w_1)<l(uw_1) )\iff m_G^u(gh) = 1\, (\text{i.e. } l(s_{\ell_1}\ldots s_{\ell_{j}}w_1)< l(s_{\ell_1}\ldots s_{\ell_{j}}uw_1))$$
    and
     $$m_G^u(h) = -1 \,(\text{i.e. } l(w_1)>l(uw_1) )\iff m_G^u(gh) = -1\, (\text{i.e. }l(s_{\ell_1}\ldots s_{\ell_{j}}w_1)> l(s_{\ell_1}\ldots s_{\ell_{j}}uw_1)).$$
   Therefore, $m_G^u(h) = m_G^u(gh)$. By applying \ref{prop; Lp-bounded-introduction}, we conclude that $T_{m_G^u}$ is $L_p$-bounded for every $1< p< \infty$.
\end{proof}

\begin{rem}
In the special case $(\Delta,\delta)=(W,\delta_W)$, where $W$ is a Coxeter group satisfying \eqref{nested-condition} relative to $u$ in \ref{theoremC}, the result reduces to \ref{theoremA}.

\end{rem}

\medskip
When considering groups acting on buildings,  we obtain the following result connecting
the existence of global fixed chambers with the form of the Fourier multiplier. We say that the multiplier defined in \eqref{eq: Def of symbol for actions on buildings} is trivial if it is
constant on $G\setminus G_0$.

\begin{prop} 
    Let $(\Delta, \delta)$, $C_0$, $G$, $G_0$ and $u$ be as in \ref{theoremC}.
    If $G$ has a global fixed chamber $C'_0$ (that is, for any $g\in G$, $gC'_0 = C'_0$), then the multiplier $m_G^u$ defined in \eqref{eq: Def of symbol for actions on buildings} is trivial on $G$.
\end{prop}

\begin{proof}
We denote $m_G^u$ simply by $m$.
    Assume that $G$ has a global fixed chamber $C'_0$. It is clear that $m$ is trivial when $C_0 = C'_0$, or when $C_0\ne C'_0$ but $G_0 =G$. Hence we may assume that $C_0\neq C'_0$ and $G_0\neq G$. Suppose that there exists $g\in G\setminus G_0$ such  that $\widetilde m(C'_0) \ne \widetilde m(gC_0)$. Consider the case $\widetilde m(C'_0)=1$ and $\widetilde m(gC_0)=-1$; the opposite case is analogous.
    
    Let $\Sigma = \phi(\Sigma(W,S))$ and $\Sigma_1 = \phi_1(\Sigma(W,S))$ be apartments containing $C'_0, C_0$ and $C_0,gC_0$, respectively, with $\phi(wK) = C'_0$, $\phi(K) = C_0 = \phi_1(K)$ and $\phi_1(w_1K) = gC_0$. Then we have
    \begin{equation}\label{equation; glbal-fixed-chamber-the-first}
       \delta(C'_0,C_0) = w^{-1}\quad \text{and} \quad\delta(C_0,gC_0) = w_1. 
    \end{equation}
   Let $N_u \subseteq W$ be the set of elements whose reduced expressions neither start nor end with $u$, and let $T_u = \{s\in S: \mathrm{m}_{su} = 2\}$.   Since $C'_0$ is a global fixed chamber, it follows from the first equation in \eqref{equation; glbal-fixed-chamber-the-first} that 
    \begin{equation}\label{equation; glbal-fixed-chamber-the-second}
    \delta(C'_0,gC_0)=\delta(g^{-1}C'_0,C_0) = \delta(C'_0,C_0) = w^{-1}.
    \end{equation}
Moreover, $\widetilde m(C'_0) = 1$ and $\widetilde m(gC_0) = -1$ imply that $C_0'\in \phi(H_u^+)$ and $gC_0\in \phi_1(H_u^-)$, respectively. Hence 
$w\in H_u^+$ and $w_1\in H_u^-$. It follows from Proposition \ref{prop;nested-condition} and Proposition \ref{prop; equi-char-Nest-algebraic-forms} that reduced expressions of $w^{-1}$ are one of the following forms:
$$w^{-1} = h,\quad w^{-1}=uh_1\quad \text{or}\quad w^{-1} \in W_{T_u},$$
and reduced expressions of $w_1$ are one of the following forms:
$$w_1 = u,\quad w_1 = uh_2\quad \text{or}\quad w_1 =uh_3u,$$
where $h,h_i\in N_u\setminus W_{T_u}$ for $i=1,2,3$. We only consider the case  $w_1 = uh_2$, where $h_2=s_1\cdots s_k$ is a reduced expression, since the other cases are analogous.
Since reduced expressions of $w^{-1}$ do not end with $u$, we have $l(w^{-1}u)>l(w^{-1})$.
Let $ C_1=\phi_1(uK)$. Then $\delta(C'_0,C_1) = w^{-1}u=\delta(C'_0,C_0)\delta(C_0, C_1)$. By a similar argument as in Proposition \ref{pro;building-transitivity}(i) and by induction,  we obtain $\delta(C'_0,gC_0) = w^{-1}us_{\ell_1}\ldots s_{\ell_j}$ for some $j\in\{1,\ldots,k\}$ and $1\le \ell_1\le \ldots\le \ell_j\le k$. This implies that 
$$\delta(C'_0,gC_0) = w^{-1}us_{\ell_1}\ldots s_{\ell_j}\ne w^{-1},$$
    which contradicts \eqref{equation; glbal-fixed-chamber-the-second}.
Hence
		$\widetilde m(C'_0) = \widetilde m(gC_0)$ for any $g\in G\setminus G_0$, i.e. $m$ is constant on $G\setminus G_0$.



\end{proof}

\subsection{A finer model for Hilbert transforms}
The multiplier defined in \eqref{eq: Def of symbol for actions on buildings} only takes the values $1$ and $-1$, since the model used in \eqref{eq: Def of symbol for actions on buildings} only distinguishes the positive and negative roots in a certain apartment. 
In what follows, we show that, under additional assumptions on a building satisfying the nested condition, one can construct a finer model to define the associated multiplier. In particular, this new model recovers the multiplier arising from groups admitting non-trivial actions on simplicial trees  in \cite{gonzalez2022non-commutative}.

Let $(\Delta, \delta)$ be a building of type $(W,S)$. Assume that there is a wall $H_0\in \Delta$ such that when removing $H_0$ from the building, the complement decomposes into  $n$ connected components for some $n\in \mathbb{N}_{\ge2}\cup \infty$, that is,
\begin{equation}\label{eq: assump on buildings}
 \Delta\setminus H_0=\bigsqcup_{i=0}^{n-1}\Delta_i.   
\end{equation} 
Here \eqref{eq: assump on buildings} gives the decomposition of the chamber set $\Delta$ into the connected components determined by $H_0$; in particular, $H_0$ is incident to each such component.
Before studying groups acting on such buildings, we first establish several properties of buildings that satisfy the above assumption.  The proposition below will also play an important role in the proof of the main theorem of this subsection. 

\begin{prop}\label{prop: properties of our special building}
 Let $(\Delta, \delta)$ be a building satisfying \eqref{eq: assump on buildings}. Fix a chamber $C_0$ whose $u$-face lies in the wall $H_0$. Then the following hold: 
 \begin{enumerate}
 \item[(i)] For any two chambers $C$ and $D$ in different connected components of $\Delta\setminus H_0$, every apartment containing $C,D$ contains the wall $H_0$. 
 \item[(ii)] For every connected component of $\Delta\setminus H_0$, there is a unique chamber $u$-equivalent to $C_0$; equivalently, the $u$-face of the unique chamber lies in $H_0$.
 \item[(iii)] The wall $H_0$ is the only wall in any apartment containing the $u$-face of $C_0$.
 \end{enumerate}
\end{prop}
    
\begin{proof}
    (i) This follows immediately from \eqref{eq: assump on buildings}.
 \medskip
 
    (ii) By the construction of the connected components in \eqref{eq: assump on buildings}, there exists a chamber in each connected component that is $u$-equivalent to $C_0$. Suppose that, for some $i\in\{0,1,\ldots,n-1\}$, there are two distinct such chambers in the connected component $\Delta_i$. Denote them by $u_{(i)}C_0$ and $u_{(i)}'C_0$. Since these two chambers are $u$-adjacent, there exists an apartment containing both of them and their common $u$-face lies in a wall $H\ne H_0$. Let $\alpha$ be a root with wall $H$ in this apartment. Without loss of generality, assume that $\alpha$ contains $u_{(i)}C_0$. Take another chamber in the connected component $\Delta_{j}$ with $i\ne j$ that is $u$-equivalent to $C_0$, and denote it by $u_{(j)}C_0$. It follows from \cite[Proposition 3.3]{abramenko2010intersections} that there is an apartment $\Sigma_{i,j}$ containing $\alpha\cup u_{(j)}C_0$. By the uniqueness of the wall which contains a fixed $u$-face in apartments (see Remark \ref{rem; unique-wall}), we conclude that $\Sigma_{i,j}$ does not contain $H_0$, which contradicts part (i). 

\medskip
    (iii) Let $\mathcal{F}$ be the set of all chambers that share the same $u$-face with $C_0$. It follows from part (ii) that $\mathcal{F}$ has exactly $n$ chambers. Then any apartment containing the $u$-face of $C_0$ must contain a chamber in $\mathcal{F}$. 
    Suppose that there is an apartment containing $C\in \mathcal{F}$ such that the $u$-face of $C$ (the same as the $u$-face of $C_0$) lies in a wall $H'\ne H_0$. Take the root $\alpha'$ that contains $C$ and the wall $H'$. Consider a connected component $\Delta_{i_0}$ such that $C\notin \Delta_{i_0}$, and let $u_{(i_0)}C_0$ be the chamber in $\mathcal{F}\cap \Delta_{i_0}$.
    By \cite[Proposition 3.3]{abramenko2010intersections} again, there is an apartment $\Sigma_{i_0}'$ containing $\alpha'\cup u_{(i_0)}C_0$. Applying the uniqueness of the wall which contains a fixed $u$-face in apartments again, $\Sigma_{i_0}'$ does not contain $H_0$, which contradicts part (i). 
\end{proof}

From now on, let $(\Delta, \delta)$ be a building of type $(W,S)$  satisfying \eqref{eq: assump on buildings} and also satisfying the nested condition, that is, $W$ satisfies \eqref{nested-condition} relative to a certain generator, say $u\in S$.
Let $G$ be a group acting on $(\Delta, \delta)$ isometrically. Fix a chamber $C_0$ whose $u$-face lies in the wall $H_0$. Next we define a multiplier $m$ on $G$ as follows. 

\begin{defn} \label{defn; symbol-tree}
Let $G$, $(\Delta, \delta)$ and $C_0$ be defined as above, 
and let $G_0 := \{g\in G: \delta(gC_0, C_0)\in W_T\}$, where $T:= \{s\in S: su=us\}$. 
Let $\widetilde{m}:\Delta \rightarrow \mathbb{C}$ be a bounded function such that 
\begin{enumerate}
    \item[(i)] $\widetilde{m}$ restricted to $\Delta \setminus H_0$, is constant on each connected component. 
    \item[(ii)] $\widetilde{m}|_{\Delta_i} = \widetilde{m}|_{\Delta_j}$ for $i,j\in \{0,1,\ldots,n-1\}$, whenever  there is $g\in G_0$ such that $g\cdot \Delta_i = \Delta_j$. 
\end{enumerate}
The associated multiplier $m$ on $G$ is defined by 
$$
m(g)=\widetilde{m}(gC_0), \; \forall g\in G.
$$
\end{defn}
Note that $G_0$ is a subgroup of $G$ by a similar argument to that in the proof of \ref{theoremC}(i). In addition, $G_0$ is the stabilizer of the wall $H_0$, that is, $G_0 = {\rm Stab}_{ H_0}$. To see this, note that $W_T = {\rm Stab}_{ H_u}$ which follows from the assumption that $W$ satisfies the nested condition relative to $u$ and Proposition \ref{prop; equi-char-Nest-algebraic-forms}. For any $g\in G$, there is an apartment $\Sigma_g = \phi_g(\Sigma(W,S))$ containing $gC_0$, $C_0$ such that $\phi_g(K) = C_0$. It follows from Proposition \ref{prop: properties of our special building}(iii) that $H_0\in \Sigma_g$ and then $H_0 = \phi_g(H_u)$. Suppose that $\phi_g(wK) = gC_0$ for some $w\in W$. Then $g\in G_0$ if and only if $w\in W_T = {\rm Stab}_{ H_u}$ if and only if $\phi_g(w\cdot H_u) = H_0$. Since the $u$-face of $gC_0$ lies in both $\phi_g(w\cdot H_u) = H_0$ and $g\cdot H_0$, Proposition \ref{prop: properties of our special building}(iii) implies that $g\cdot H_0 = H_0$ for $g\in G_0$.

Now we show that the multiplier defined above also admits an algebraic form in terms of the distance between chambers in the building. 
By  Proposition \ref{prop: properties of our special building}(ii), each component of $\Delta \setminus H_0$ contains a unique chamber $u$-equivalent to $C_0$.
Denote these chambers by $u_{(i)}C_0$ for $i\in \{0, 1, \ldots, n-1\}$  and identify $C_0$ with $u_{(0)}C_0$. 

Let $D\in \Delta$. If $D$ and $C_0$ are in the same connected component, by Proposition \ref{prop: properties of our special building}(iii) there exists an apartment $\Sigma$ containing $D$, $C_0$ and $H_0$. 
We claim that $D$ and $C_0$ lie in the same root with wall $H_0$ in this apartment. Suppose that they are in different roots with wall $H_0$. Consider the chamber $u_{\Sigma}C_0$, which is $u$-adjacent to $C_0$ and lies in the same root as $D$. 
Then $C_0$ and $u_{\Sigma}C_0$ are two distinct chambers in the same component, which contradicts Proposition \ref{prop: properties of our special building}(ii). 
Therefore, $D$ and $C_0$ are in the same root with the wall $H_0$ in $\Sigma$. 
 Choose $\phi$ such that $\Sigma = \phi(\Sigma(W,S))$ and $\phi(K) = C_0$. Then the root $\alpha=\phi(H_u^+)$ contains $D, C_0$ and $H_0$ (recall that the root $H_u^+$ in the Coxeter complex contains $K$). By \cite[Proposition 3.3]{abramenko2010intersections}, there exists an apartment $\Sigma'_i = \phi_i'(\Sigma(W,S))$ with $\phi_i'(K) = C_0$ containing $\alpha \cup \{u_{(i)}C_0\}$. Therefore, we have 
$$d(D,C_0)< d(D, u_{(i)}C_0), \; \forall \; 1\leq i\leq n-1.$$
On the other hand, if $D$ is in the same connected component  as some $u_{(i)}C_0$ with $1\leq i\leq n-1$, then by a similar argument as above, there exists a root $\alpha_i = \phi_{i}(H_u^+)$ with $\phi_i(K) = u_{(i)}C_0$ such that it contains $D$ and $u_{(i)}C_0$. Consequently,
$$d(D, u_{(i)}C_0)< d(D, C_0).$$

By the discussion above,  the multiplier $m$ in Definition \ref{defn; symbol-tree} can be reformulated as follows. For each $i\in \{0, 1,\ldots, n-1\}$, denote by $\mathbf{C}_i\in \mathbb{C}$ the value of $\widetilde{m}$ on $\Delta_i$, and let $u_{(i)}C_0$ be the chambers chosen as above. For any $g\in G$,
        \begin{equation*}
		m(g) = \left\{\begin{array}{l}
			\mathbf{C}_0, \quad \text{if}\; d(gC_0, C_0) < d(gC_0, u_{(1)}C_0),\; \\
			\mathbf{C}_i, \quad \text{if} \; d(gC_0, C_0)> d(gC_0, u_{(i)}C_0), \;\;\text{ for } 1\leq i\leq n-1.
		\end{array} \right.
		\end{equation*}


Before verifying that the multiplier $m$ in Definition \ref{defn; symbol-tree} satisfies the Cotlar identity, we first prove a lemma showing that the fixed chamber $C_0$ in Definition~\ref{defn; symbol-tree} can be replaced by any chamber that is $u$-adjacent to $C_0$ without changing the value of $m(g)$ for any $g\notin G_0$. 

\begin{lem}\label{lem: same cc}
Let $G$, $G_0$, $(\Delta, \delta)$, $C_0$ and $H_0$ be defined as above. For each $j\in \{1,\ldots, n-1\}$, let $u_{(j)}C_0$ be the chamber $u$-adjacent to $C_0$ lying in the $j$-th connected component $\Delta_j$ of $\Delta\setminus H_0$. Then for any $j\in \{0,1,\ldots, n-1\}$ and any $g\notin G_0$,
the chambers $gu_{(j)}C_0$ and $gC_0$ lie in the same connected component  of $\Delta\setminus H_0$. 
\end{lem}
\begin{proof}

For a given $g$, there exists some $i\in \{0,1,\ldots, n-1\}$ such that $gC_0$ and $u_{(i)}C_0$ are in the same component. By the discussion preceding this lemma, there is a root $\beta$ containing $gC_0$ and $u_{(i)}C_0$ with wall $H_0$. Since $u_{(j)}C_0$ and $u_{(i)}C_0$ are $u$-adjacent and their common $u$-face lies in $H_0$, by \cite[Proposition 3.3]{abramenko2010intersections}, there is an apartment $\Sigma_1$ containing $\beta\cup u_{(j)}C_0$.  Consider $g^{-1}$ acting on $\Sigma_1$, then the apartment $g^{-1}\cdot \Sigma_1$ contains $C_0$, $g^{-1}u_{(i)}C_0$, $g^{-1}u_{(j)}C_0$ and the wall $g^{-1}\cdot H_0$. 
Since the group action preserves the $u$-adjacency and 
maps walls to walls, 
$g^{-1}u_{(i)}C_0$ and $g^{-1}u_{(j)}C_0$ are $u$-adjacent, and their common $u$-face lies in $g^{-1}\cdot H_0$. By Proposition \ref{prop: properties of our special building}(iii), the $u$-face of $C_0$ only lies in the wall $H_0$, thus $H_0$ is also in the apartment $g^{-1}\cdot \Sigma_1$. Since $g\notin G_0={\rm Stab}_{H_0}$, we have $g^{-1}\cdot H_0\ne H_0$. By the uniqueness of the wall containing a fixed $u$-face in an apartment, the common $u$-face of $g^{-1}u_{(i)}C_0$ and $g^{-1}u_{(j)}C_0$ does not lie in $H_0$. Therefore, these two chambers lie in the same connected component of $\Delta\setminus H_0$. 
If $i=0$, it follows that $g^{-1}u_{(j)}C_0$ and $g^{-1}C_0$ are in the same component. 
If $i\ne 0$, then $C_0\notin \beta$. In this case, let $\Sigma_2$ be an apartment containing $\beta\cup\{C_0\}$. 
An argument analogous to the one for $\Sigma_1$  shows that $g^{-1}u_{(i)}C_0$ and $g^{-1}C_0$ are in the same component, which implies that $g^{-1}u_{(j)}C_0$ and $g^{-1}C_0$ are in the same component.    
\end{proof}

For any $g\in G$ with $\delta(gC_0,C_0) = w$,  it follows from the discussion below Definition \ref{defn; symbol-tree} that $g\in G_0\Longleftrightarrow w\in W_T\Longleftrightarrow w\cdot H_u^+\in \{H_u^+, H_u^-\}$. Therefore, by axiom (B2) in Definition~\ref{def: Buildings}, we obtain that $\delta(gC_0,u_{(i)}C_0)\cdot H_u^+ \in \{H_u^+, H_u^-\}$ for any $0\le i\le n-1$.

\begin{thm}\label{thm; generating-trees}
	Let $(\Delta, \delta)$ be a building of type $(W,S)$, where $W$ satisfies \eqref{nested-condition} relative to $u$. Let $C_0$, $G$, $G_0$ and $m$ be the same as above. 
    Then the following statements hold:
    \begin{enumerate}
        \item[(i)]  $m$ satisfies \eqref{eq: Cotlar}, i.e.
	$$(m(g)-m(h))(m(g^{-1}h)-m(g^{-1})) = 0, \text{ for any }g\in G\setminus G_0 \text{ and } h\in G.$$

    \item[(ii)] The multiplier $m$ is left $G_0$-invariant, and then  $T_m$ gives an $L_p$-bounded Fourier multiplier on $L_p(\mathcal{L} G)$ for $1< p< \infty$.
    \end{enumerate}
\end{thm}
\begin{proof}
	(i)  Suppose that $m(g) \ne m(h)$. 
    Assume that $m(g) = \mathbf{C}_{i}$ and $m(h) = \mathbf{C}_j$ for some  $i,j\in \{0,1,\ldots,n-1\}$ with $i\ne j$. There is an apartment $\Sigma = \phi(\Sigma(W,S)) $ of $\Delta$ such that $gC_0,hC_0\in \Sigma$. Since $gC_0$ and $hC_0$ are in different connected components, Proposition \ref{prop: properties of our special building}(i) implies that $\Sigma$ contains the wall $H_0$. Since the chamber $u_{(j)}C_0$ is the unique chamber in $\Delta_j$ whose $u$-face lies in $H_0$, it follows that $u_{(j)}C_0$ is also in $\Sigma$. Moreover, $hC_0$ and $u_{(j)}C_0$ are in the same root with wall $H_0$ in $\Sigma$ by the discussion preceding Lemma \ref{lem: same cc}. Choose $\phi$ such that $\phi(K) = u_{(j)}C_0$. Then $H_0 = \phi(H_u)$, $hC_0$ and $u_{(j)}C_0$ are in the root $\phi(H_u^+)$, while $gC_0$ is in $\phi(H_u^-)$.
   Suppose that $\phi(wK) = gC_0$ for some $w\in W$, then $wK\in H_u^-$. Since $W$ satisfies \eqref{nested-condition} relative to $u$, Proposition \ref{prop;nested-condition}, the assumption $g\notin G_0$ (equivalently, $w\notin W_T$) and the above argument before this theorem ensure that
    $$H_u^+\subsetneq w\cdot H_u^+\quad \text{or}\quad H_u^+\subsetneq w\cdot H_u^-.$$
    Then we obtain the following inclusion relation between roots:
    \begin{equation*}\label{relation;proof-trees}
        \phi(H_u^+)\subsetneq \phi(w\cdot H_u^+)\quad \text{or}\quad \phi(H_u^+)\subsetneq \phi(w\cdot H_u^-).
    \end{equation*}
    This implies that $hC_0$ and $u_{(j)}C_0$ are in the same root with wall $\phi(w\cdot H_u)=g\cdot H_0$, and note that the $u$-face of $gC_0$ is contained in this wall.
Consider the action of $g^{-1}$ on $\Sigma$. 
Since the group action is an isometry, 
it follows that $g^{-1}hC_0$ and $g^{-1}u_{(j)}C_0$ are in the same root with the wall $H_0$. Thus $g^{-1}hC_0$ and $g^{-1}u_{(j)}C_0$ are in the same connected component of  $\Delta\setminus H_0$. By Lemma \ref{lem: same cc}, $g^{-1}u_{(j)}C_0$ and $g^{-1}C_0$ are  in the same connected component for any $g\notin G_0$.
 Therefore, $g^{-1}C_0$ and $g^{-1}hC_0$ are in the same connected component of $\Delta\setminus H_0$, which implies that $m(g^{-1}) = m(g^{-1}h)$.

\medskip

    (ii) We first claim that the second condition of $\widetilde{m}$ in Definition \ref{defn; symbol-tree} implies that $\widetilde{m}$ is invariant under the action of $G_0$ on $\Delta \setminus H_0$, that is, if $g\in G_0$ and a chamber $D\in \Delta_i$ satisfy $gD\in \Delta_j$ for some $i,j\in \{0,\ldots,n-1\}$, then  $gE\in \Delta_j$ for every $E\in\Delta_i$. It follows directly from this claim that $m$ is left $G_0$-invariant, that is, for any $g\in G_0$ and $h\in G$, one has $m(gh) = m(h)$. Now we prove the claim. Assume on the contrary that there are $D,E\in \Delta_i$ and $g\in G_0$ such that $gD\in \Delta_j$ but $gE\notin \Delta_j$. Let $\Sigma'$ be an apartment containing $D$ and $E$, and let $\mathcal{U} = \{H\in \Sigma': H \text{ is a wall separating } D \text{ and } E\}$. Since $D$ and $E$ are in the same component $\Delta_i$, it follows from the discussion preceding Lemma \ref{lem: same cc} that $H_0\notin \mathcal{U}$.
Consider $g$ acting on $\Sigma'$, then $g\cdot \Sigma'$ contains the chambers $gD$ and $gE$. Since $gE$ and $gD$ are in different components, Proposition \ref{prop: properties of our special building}(i) implies that the apartment $g\cdot \Sigma'$ contains $H_0$, which separates $gD$ and $gE$. This implies that there exists $H\in \mathcal{U}$ such that $g\cdot H = H_0$. Since $g\in G_0$ (that is, $g$ is in the stabilizer of $H_0$), we obtain $H_0= H\in \mathcal{U}$, which is a 
contradiction. By applying \ref{prop; Lp-bounded-introduction}, we conclude that $T_m$ is $L_p$-bounded for every $1< p< \infty$.
\end{proof}

 Now we consider the case $S =\{s,t\}$ and $W = D_{\infty} := \langle S:  s^2 = t^2 = 1\rangle$, the infinite dihedral group, which is a right-angled Coxeter group. Any building $(\Delta,\delta)$ of type $(D_{\infty}, S)$ is a tree with no vertices of valence 1 \cite[Example 6.5]{thomas2018geometric}, and it satisfies assumption \eqref{eq: assump on buildings}. 
 In this case, it is easy to see that the Coxeter complex of $D_{\infty}$ is a bi-infinite tree as follows:
	\begin{center}
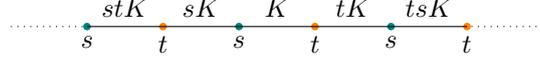

		
		\begin{tikzpicture}
			\coordinate[label = below: $s$] (s) at (0,0);
			\coordinate[label = below: $t$] (t) at (1,0);
			\coordinate[label = below: $t$] (s1) at (-1,0);
			\coordinate[label = below: $s$] (s2) at (-2,0);
			\coordinate[label = below: ] (s3) at (-3,0);
			\coordinate[label = below: $s$] (t1) at (2,0);
			\coordinate[label = below: $t$] (t2) at (3,0);
			\coordinate[label = below: ] (t3) at (4,0);

			\node at (s)[circle, fill = teal, inner sep=1pt]{};
			\node at (t)[circle, fill = orange, inner sep=1pt]{};
			\node at (s1)[circle, fill = orange, inner sep=1pt]{};
			\node at (s2)[circle, fill = teal, inner sep=1pt]{};
			\node at (t1)[circle, fill = teal, inner sep=1pt]{};
			\node at (t2)[circle, fill = orange, inner sep=1pt]{};

			\draw [dotted] (s3)--(s2) node[above, midway]{};
			\draw (s1)--(s2) node[above, midway]{$stK$};
			\draw (s1)--(s) node[above, midway]{$sK$};
			\draw (s)--(t) node[above, midway]{$K$};
			\draw (t1)--(t) node[above, midway]{$tK$};
			\draw (t1)--(t2) node[above, midway]{$tsK$};
			\draw [dotted] (t3)--(t2) node[above, midway]{};
		\end{tikzpicture} 
         \captionof{figure}{The Coxeter complex of  $D_{\infty}$.}
	\end{center}
The fundamental chamber $K$ is an edge with endpoints labeled by $s$ and $t$. From the Coxeter complex of $D_{\infty}$, it is clear that every chamber in $\Delta$ is an edge and every vertex of a chamber is both a wall and a face. 
 Let $x_0$ be a vertex of valence-$n$, where $n\in \mathbb{N}_{\ge2}\cup \infty$. 
Let $G$ be a group acting on $(\Delta, \delta)$ isometrically. 
  When cutting the building along the wall (or vertex) $x_0$, 
  we obtain $n$ connected components. The subgroup $G_0$ in this case is ${\rm Stab}_{x_0}$. The chamber $C_0$ can be chosen to be any edge having $x_0$ as a vertex, and $gC_0$ lies in the same connected component as $gx_0$ for any $g\in G\setminus G_0$. 
  Then the multiplier defined in Definition \ref{defn; symbol-tree} recovers the one given in \cite{gonzalez2022non-commutative} via the model \eqref{eq: Multiplier on trees} for groups acting on simplicial trees.


	
	


\subsection{Examples of groups acting on right-angled buildings}\label{subsection3.5}
Buildings associated with right-angled Coxeter groups are called right-angled buildings. In particular, they satisfy the nested condition, so that \ref{theoremC} applies.
In this subsection, we present some examples of groups admitting non-trivial actions on right-angled  buildings  and we discuss the concrete forms of the corresponding Hilbert transforms on these groups.

 There is a large class of groups called \textit{graph products of groups}, which act naturally on their associated right-angled buildings. 


\begin{defn}\label{def;graph product of groups}
	Let $\Gamma = (S, E(\Gamma))$ be a finite simplicial graph with vertex set $S$ and edge set $E(\Gamma)$. Let every vertex $s\in S$ be associated with a non-trivial group $G_s$. The \textit{graph product} of the family $\{G_s: s\in S\}$ is the group $G_{\Gamma} := (\underset{s\in S}{\ast}G_s)/N$, where $N$ is the normal subgroup generated by the set of commutators $\{g_tg_sg_t^{-1}g_s^{-1}: g_t\in G_t, g_s\in G_s\; \text{and}\; \{t,s\}\in E(\Gamma)\}$. 
\end{defn}
Note that two vertex groups  $G_s$ and $G_t$ commute in $G_{\Gamma}$ if and only if the corresponding vertices $s$ and $t$ are adjacent in $\Gamma$. Graph products of groups are closely related to Coxeter groups. 
Recall that Coxeter groups can be defined by their \textit{presenting graphs}, which are finite simplicial graphs $\Gamma = (S, E(\Gamma))$  with vertex set $S$ and edge set $E(\Gamma)$ such that every edge $\{s,t\}\in E(\Gamma)$ is labeled by an integer $\mathrm{m}_{st}\ge 2$.  When $\mathrm{m}_{st} = 2$, the label is usually omitted. 
 Then the associated Coxeter group $W_{\Gamma}$ is defined by 
$$W_{\Gamma} := \langle s\in S: s^2 = 1,\;  (st)^{\mathrm{m}_{st}} = 1\; \text{if}\; \{s,t\}\in E(\Gamma) \rangle.$$
Graph products of groups were introduced by Green \cite{green1990graph} in her PhD thesis. She showed that, for any graph product of groups $G_{\Gamma}$ and any $1\ne g\in G_{\Gamma}$, $g$ admits a \textit{reduced expression} $g = g_{i_1}\ldots g_{i_k}$, where $1\ne g_{i_j}\in G_{s_{i_j}}$ for $j\in \{1,\ldots ,k\}$. Moreover, $s_{i_1}\ldots s_{i_k}$ is a reduced expression for an element $w\in W_{\Gamma}$. Since $\Gamma$ is unlabeled, we have $\mathrm{m}_{st}\in\{1,2,\infty\}$, and then  $W_{\Gamma}$ is a right-angled Coxeter group. 

 Recall that a \textit{full subgraph} of a simplicial graph $\Gamma = (S,E(\Gamma))$ is a subgraph $\Gamma'=(S',E(\Gamma'))$ 
 that inherits every edge $\{s,t\}$ in $E(\Gamma)$ whenever $s,t\in S'$. In this case, $G_{\Gamma'}$ is naturally a subgroup of $G_{\Gamma}$.  
As explained in \cite[Lemma 3.20]{green1990graph}, graph products of groups can be realized as free products, direct products or amalgamated free products of groups.
We summarize this lemma in the following form.

\begin{lem}(\cite[Lemma 3.20]{green1990graph})
    Let $\Gamma = (S, E(\Gamma))$ be a  finite simplicial graph, and let $G_{\Gamma}$ be the graph product of the family $\{G_s: s\in S\}$. Then $G_{\Gamma}$ admits the following amalgamated free product decomposition:
    \begin{equation}\label{eq;graph-products-amalgmated prodcuts}
    G_{\Gamma} = (G_s\times G_{\Gamma_1})\ast_{G_{\Gamma_1}}G_{\Gamma_2},
    \end{equation}
    where  $s$ is any vertex in $S$, $\Gamma_1$ is the full subgraph generated by the vertices adjacent to $s$, and $\Gamma_2$ is the full subgraph generated by $S\setminus \{s\}$.
\end{lem}
Note that if $s$ is an isolated vertex (that is, there are no edges joining $s$ to any other vertex in $S$), then the decomposition of $G_{\Gamma}$ in \eqref{eq;graph-products-amalgmated prodcuts} reduces to a free product of two subgroups. 
On the other hand, if $\{s,t\}\in E(\Gamma)$ for every $t\in S\setminus\{s\}$, then the decomposition in \eqref{eq;graph-products-amalgmated prodcuts} reduces to a direct product of two subgroups.
We illustrate these with the following canonical examples of graph products of groups.

\begin{ex}
	(i) Consider $G_s = \mathbb{Z}_2$ for all $s\in S$. Then $G_{\Gamma} = W_{\Gamma}$ is the right-angled Coxeter group with presenting graph $\Gamma$.
	
	(ii) Consider $G_s = \mathbb{Z}$ for all $s\in S$. Then $G_{\Gamma} = A_{\Gamma}$ is the right-angled Artin group with presenting graph $\Gamma$:
	 $$A_\Gamma = \langle s\in S: st = ts \; \text{if}\; \{s,t\}\in E(\Gamma)\rangle.$$ 
    In particular, if for any $s,t\in S$, $\{s,t\}\in E(\Gamma)$, then the corresponding right-angled Artin group $A_\Gamma$ is a direct product of copies of $\mathbb{Z}$. If $E(\Gamma)=\emptyset$, then $A_\Gamma$ is a free product of copies of $\mathbb{Z}$.
\end{ex}

Let $\Gamma$ be a finite simplicial graph with vertex set $S = \{s_i: i\in I\}$, let $(W_{\Gamma},S)$ be the right-angled Coxeter system with presenting graph $\Gamma$, and let $G_{\Gamma}$ be a graph product of groups over $\Gamma$. As mentioned in \cite[Section 7.5]{thomas2018geometric}, 
one can construct a right-angled building $(G_{\Gamma}, \delta)$ of type $(W_{\Gamma}, S)$ with  $\delta: G_{\Gamma}\times G_{\Gamma} \to W_{\Gamma}$ defined by 
\begin{equation}\label{eq: Def for the distance on graph products}
   \delta(g, g') := s_{i_1}\ldots s_{i_k}, 
\end{equation}
where $g, g'\in G_{\Gamma}$ and $g^{-1}g'$ admits a reduced expression $g_{i_1}\ldots g_{i_k}$. One can check that this is a well-defined building by applying \cite[Theorem 3.9]{green1990graph}.
Consider the action of $G_{\Gamma}$  on the building $(G_{\Gamma}, \delta)$ by left multiplication. It is straightforward to verify that this action is simply transitive on the set of chambers and preserves the distance $\delta$.


Fix a chamber $C_0:= 1\in G_{\Gamma}$. Let $u\in S$ be a generator (since $W_\Gamma$ is right-angled, $u$ may be chosen arbitrarily in $S$) and let $m$ be the multiplier on $G_{\Gamma}$ defined by \eqref{eq: Def of symbol for actions on buildings}, whose algebraic form is given by \eqref{eq;Def of algebric form of m on buildings}. For any $1\ne h\in G_{\Gamma}$ with a reduced expression $h_{i_1}\dots h_{i_n}$, the definition of $\delta$ on $G_{\Gamma}$ in \eqref{eq: Def for the distance on graph products} yields that $\delta(C_0,hC_0) = s_{i_1}\dots s_{i_n}$. Denote $s_{i_1}\dots s_{i_n}$
by $w$, then 
\begin{equation}\label{eq: symbol on G Gamma}
 m(h) =  \left\{\begin{array}{l}
		1, \quad \text{if} \; l(w)<l(uw),\\
		-1, \quad \text{if} \;  l(w)>l(uw).
	\end{array} \right.   
\end{equation}
Equivalently, 
\begin{equation}\label{eq: algebraic symbol on G Gamma}
 m(h) =  \left\{\begin{array}{l}
		1, \quad \text{if there is no reduced expression of $h$ starting with a letter in $G_u$}, \\
		-1, \quad \text{if  there is a reduced expression of $h$ starting with a letter   in $G_u$}.
	\end{array} \right.   
\end{equation}

Now we compare the multiplier $m$ on $G_{\Gamma}$ defined  above  with that in \cite{gonzalez2022non-commutative} by \eqref{eq: Multiplier on trees} when viewing $G_{\Gamma}$ as an amalgamated free product through the decomposition \eqref{eq;graph-products-amalgmated prodcuts}.
Recall that for an amalgamated free product $\ast_{j,A} G_j$, choose a right coset representative $E_j$ of $G_j$ modulo $A$, then
every element $g\in \ast_{j,A} G_j$ admits a unique normal form
$$g = ag_{1}\dots g_{k},$$
where $a\in A$, and each $g_{\ell}$ for $1\leq \ell\leq k$ is in some $E_j$ such that $ g_{\ell},g_{\ell+1}$ are not in the same factor $E_j$. The multiplier $m'$ on the amalgamated free product $\ast_{j,A} G_j$ introduced in \cite{gonzalez2022non-commutative} 
is defined as follows. Let $\mathbf{C}_j\in \mathbb{C}$. For any $g\in (\ast_{j,A} G_j)\setminus A$ with a unique normal form $ag_{1}\dots g_{k}$, set
$$m'(g) := \mathbf{C}_j,$$
if $g_1\in G_j\setminus A$. For any $g\in A$, define $m'(g) := \mathbf{C}_{j_0}$, where $j_0$ is an arbitrarily fixed index.


\begin{prop}
    Let $(G_{\Gamma},\delta)$, $C_0$, $m$ and $m'$ be as above. Then the multipliers $m$ and $m'$ on $G_\Gamma$ coincide.
\end{prop}
\begin{proof}
    Let $s$ be a vertex of $\Gamma$ and set 
     $T_s:= \{t\in S: \mathrm{m}_{ts} =2\}$. Consider the decomposition \eqref{eq;graph-products-amalgmated prodcuts} of $G_\Gamma$. By construction, $G_{\Gamma_1} = G_0 :=\{g\in G_{\Gamma}: \delta(C_0,gC_0)\in W_{T_s}\}$. 
     Fix a right coset representative $E_2$ of $G_{\Gamma_2}$ modulo $G_0$. For any $g\in G_{\Gamma}$, let $g = ag_{1}\dots g_{k}$ be its normal form in the amalgamated free product $(G_s\times G_0)\ast_{G_0}G_{\Gamma_2}$, where $a\in G_0$ and $g_{\ell}$ is either in $G_s$ or $ E_2$ with $1\le \ell\le k$. Then $m'$ can be rewritten as   
     $$m'(g) =  \left\{\begin{array}{l}
		1, \quad \text{if} \; g_{1}\in E_2,\\
		-1, \quad \text{if} \;  g_{1}\in G_s,
	\end{array} \right.$$
    and $m'(g)=1$ if $g\in G_0$. 
    Since $W_\Gamma$ is right-angled, we have $\mathrm{m}_{st} \in\{2, \infty\}$ for any $s\ne t\in S$. It follows that $G_{\Gamma_1}=G_0$ is the centralizer of $G_s$ in $G_\Gamma$, and $G_{\Gamma_2}=G_{\Gamma_1} \cup G_2$, where $G_2$ consists of any element which is free with all the elements in $G_s$. 
It is straightforward to see that elements of $E_2$ and that of $G_s$ are free. We now compare the multipliers $m$ and $m'$ on $G_\Gamma$. 
   Let $g\in G_{\Gamma}$. Suppose that $m(g)=-1$. By \eqref{eq: algebraic symbol on G Gamma}, there is a reduced expression $g_{i_1}\dots g_{i_n}$ of $g$ such that $g_{i_1}\in G_s$. Note that for any $t\neq s$ in $S$, we have $G_t\subseteq G_{\Gamma_2}$. Therefore, this reduced expression of $g$ automatically gives rise to a  normal form of $g$ as an element in the amalgamated free product $(G_s\times G_0)\ast_{G_0}G_{\Gamma_2}$, and this normal form starts with a letter in $G_s$. By the uniqueness of normal forms in amalgamated free products, we conclude that ${m}'(g)=-1$. Conversely, suppose that $m'(g)=-1$, that is,
   $g$ admits a unique normal form $ag_{1}\dots g_{k}$ in $(G_s\times G_0)\ast_{G_0}G_{\Gamma_2}$ with $g_1\in G_s$. By writing each $g_\ell$ ($1<\ell\leq k$) as a reduced expression in $G_\Gamma$, we obtain a (not necessarily unique) reduced expression of $g$ in $G_\Gamma$. Since elements of $E_2$ and that in $G_s$ are free, we cannot use the relators in $G_\Gamma$ to replace the starting letter $g_1$ by a letter not in $G_s$. Hence, $m(g)=-1$. 
   Therefore, we conclude that $m$ and $m'$ coincide on $G_\Gamma$.
\end{proof}

\begin{rem}
    When strengthening the group action from an isometric action to a strongly transitive action (see \cite[Definition 6.1]{abramenko2008buildings}) on buildings, the group also admits a non-trivial action on an infinite tree.  More precisely, it was shown in \cite[Proposition 2.1]{dymara2002cohomology} that if $G$ is a closed subgroup of finite covolume in ${\rm Aut}(\Delta)$ whose action on $\Delta$ is strongly transitive, then $G$ admits an action on an infinite tree without fixed points.
In this case, the problem of defining Hilbert transforms on such groups 
reduces to the case covered in \cite{gonzalez2022non-commutative}. Therefore, to construct examples of groups which admit actions on buildings satisfying the nested condition but do not admit non-trivial actions on trees, one may look for groups whose actions on buildings are not strongly transitive.
\end{rem}

\section*{Acknowledgements}
The authors gratefully acknowledge the valuable discussions with Graham Niblo about Coxeter groups. Thanks are also due to Max Carter for pointing out the result in \cite{dymara2002cohomology} mentioned in the last remark. Xiao-Qi Lu is supported by the China Scholarship Council (Grant No. 202306200023).

\bibliographystyle{plain}
\bibliography{nonc-Cotlar}
\end{document}